\documentclass[11pt, reqno]{amsart}

\usepackage{a4wide}
\usepackage{amsfonts,amsmath,amsthm,amssymb,mathrsfs}
\usepackage[english]{babel} 
\usepackage[utf8x]{inputenc}%\usepackage[applemac]{inputenc}%\usepackage[T1]{fontenc}
\usepackage{yfonts}
\usepackage{amsfonts}
\usepackage{graphicx}
\usepackage{fancyhdr}
\usepackage{verbatim}
\usepackage{stmaryrd}
\usepackage[usenames, dvipsnames,svgnames,x11names,table,fixinclude]{xcolor}
\usepackage{esint}
\usepackage{setspace}
\usepackage[colorlinks,bookmarks=false]{hyperref} %va messo per ultimo

\makeatother

\theoremstyle{plain}
\makeatletter
\newtheorem*{rep@theorem}{\rep@title}
{\end{rep@theorem}}
\makeatother

\newcommand\res{\mathop{\hbox{\vrule height 7pt width .3pt depth 0pt
\vrule height .3pt width 5pt depth 0pt}}\nolimits}
\newcommand{\R}{\mathbb{R}}

\newcommand{\Leb}[1]{{\mathscr L}^{#1}}
\newcommand{\Hardy}{\mathcal{H}^1}

\newcommand{\e}{\varepsilon}

\newcommand\sbarretta{\mathop{\hbox{\vrule height 5pt width .3pt depth 6pt}}\nolimits}% \vrule height .3pt width 0pt depth 0pt}}\nolimits}

\newcommand\N{{\bf N}}
\newcommand\Z{\mathbb{Z}}

\newcommand{\segop}{[\kern-2pt[}
\newcommand{\segcl}{]\kern-2pt]}
\newcommand{\Haus}[1]{{\mathscr H}^{#1}}

\newcommand{\Tan}[1]{{\rm Tan}^{(#1)}}

\newcommand{\rk}{{\rm rk}}
\newcommand{\DD}{\mathscr{D}}
\newcommand{\north}{\mathcal{N}}

\newcommand{\F}{\mathbf{F}}

\newcommand{\M}{\mathbf{M}}
\newcommand{\Sz}{\mathbf{S}}
\newcommand{\Lip}{{\rm Lip}}
\newcommand{\spt}{{\rm spt}}
\newcommand{\E}{{\bf E}}
\newcommand{\Or}{{\bf O}}

\newcommand{\GB}[2]{GB_{#2}V}  %questo scrive GSB_nV, senza la dimensione del dominio

\newcommand{\GSB}[2]{GSB_{#2}V}  %questo scrive GSB_nV, senza la dimensione del dominio
\newcommand{\GSBnV}{GSB_nV}
\newcommand{\loc}{{\rm loc}}

\newcommand{\w}{\rightharpoonup}
\newcommand{\wstar}{\overset{*}{\rightharpoonup}}

\newtheorem{definition}{Definition}[subsection]
\newtheorem{theorem}[definition]{Theorem}
\newtheorem{proposition}[definition]{Proposition}
\newtheorem{corollary}[definition]{Corollary}
\newtheorem{lemma}[definition]{Lemma}
\theoremstyle{definition}
\newtheorem{example}[definition]{Example}

\title[Special functions of bounded higher variation]
{Compactness of special functions of\\ bounded higher variation}

\author{Luigi Ambrosio, Francesco Ghiraldin}
\address{Scuola Normale Superiore di Pisa, Piazza dei Cavalieri 7, I-56126, Pisa, Italy}
\email{l.ambrosio@sns.it, f.ghiraldin@sns.it}

\begin{document}
\noindent
\begin{abstract}
Given an open set $\Omega\subset\R^m$ and $n>1$, we introduce the new spaces $GB_nV(\Omega)$ of {\it Generalized functions of
bounded higher variation} and $GSB_nV(\Omega)$ of {\it Generalized special functions of
bounded higher variation} that generalize, respectively, the space $B_nV$ introduced
by Jerrard and Soner in \cite{JerSon02} and the corresponding $SB_nV$ space
studied by De Lellis in \cite{Del00}. In this class of spaces, which allow as in \cite{JerSon02}
the description of singularities of codimension $n$, the distributional jacobian $Ju$ need not have
finite mass: roughly speaking, finiteness of mass is not required for the $(m-n)$-dimensional part
of $Ju$, but only finiteness of size. In the space $GSB_nV$ we are able to provide compactness of
sublevel sets and lower semicontinuity of Mumford-Shah type functionals, 
in the same spirit of the codimension 1 theory \cite{Amb89,Amb95}.
\end{abstract}
\maketitle

\section{Introduction}\label{s:intro}
\noindent
The space $BV$ of functions of bounded variation, consisting of real-valued functions $u$ defined in a domain
of $\R^m$ whose distributional
derivative $Du$ is a finite Radon measure, may contain discontinuous functions and, precisely for this reason,
can be used to model a variety of phenomena, while on the PDE side it plays an important role in the theory
of conservation laws \cite{Daf,Bre}. In more recent times, De Giorgi and the first author introduced the distinguished 
subspace $SBV$ of special functions of bounded variation, whose distributional derivative consists of an absolutely
continuous part and a singular part concentrated on a $(m-1)$-dimensional set, called (approximate) discontinuity
set $S_u$. See \cite{AmbFusPal} for a full account of the theory, whose applications include the minimization of the
Mumford-Shah functional \cite{MumSha89} and variational models in fracture mechanics.
In a vector-valued setting, also the spaces $BD$ and $SBD$ play an important role, in connection with
problems involving linearized elasticity and fracture (see also the recent work by Dal Maso on the
space $GSBD$ \cite{Dal11})

It is well know that $|Du|$ vanishes on $\Haus{n-1}$-negligible sets, hence $BV$ and all related spaces can't
be used to describe singularities of higher codimension. For this reason, having in mind application to the
Ginzburg-Landau theory (where typically singularities, e.g. line vortices in $\R^3$ have codimension 2) Jerrard and Soner
introduced in \cite{JerSon02} the space $B_nV$ of functions of bounded higher variation, where $n$ stands
for the codimension: roughly speaking
it consists of Sobolev maps $u:\Omega\to\R^n$ whose distributional Jacobian $Ju$ (well defined, at least as
a distribution, under appropriate integrability assumptions) is representable by a vector-valued measure: in this
case the natural vector space is the space $\Lambda_{m-n}\R^m$ of $(m-n)$-vectors. Remarkable extensions
of the $BV$ theory have been discovered in \cite{JerSon02}, as the counterpart of the coarea formula and of
De Giorgi's rectifiability theorem for sets of finite perimeter. Even before \cite{JerSon02}, the distributional jacobian has been studied 
in many fundamental works as \cite{Mor,Bal77,GiaModSou,Sve88,MulSpe95} in connection with variational problems
in nonlinear elasticity (where typically $m=n$ and $u$ represents a deformation map), e.g. to model cavitation effects.

As a matter of fact, since $Ju$ can be equivalently described as a flat $(m-n)$-dimensional current,
an important tool in the study of $Ju$ is the well-developed machinery
of currents, both in the Euclidean and in metric spaces, see 
\cite{Fed,Fle66,AmbKir00,Whi99rect}.
The fine structure of the measure $Ju$ has been investigates in subsequent papers: using
precisely tools from the theory of metric
currents \cite{AmbKir00}, De Lellis in \cite{Del00}
characterized $Ju$ in terms of slicing and proved rectifiability of the (measure theoretic)
support $S_u$ of the $(m-n)$-dimensional part of $Ju$, while
in \cite{DelGhi10} the second author and De Lellis characterized the absolutely continuous part of $Ju$ with respect to
$\Leb{m}$ in terms of the Sobolev gradient $\nabla u$.
Also, in \cite{Del00} the analog of the space $SBV$ has been introduced, denoted $SB_nV$: it consists
of all functions $u\in B_nV$ such that $Ju=R+T$, with $\|R\|\ll\Leb{m}$ and $\|T\|$ concentrated on a 
$(m-n)$-dimensional set. 

The main goal of this paper is to study the compactness properties of $SB_nV$.
Even in the standard $SBV$ theory, a uniform control on the energy of Mumford-Shah type
$$
\int (|u_h|^s+|\nabla u_h|^p) d\Leb{m}+\Haus{m-n}(S_{u_h})
$$
(with $s>0$, $p>1$)
along a sequence $(u_h)$ does not provide a control on $Du_h$. Indeed, only the $\Haus{m-1}$-dimensional
measure of $S_{u_h}$ does not provide a control on the width of the jump. 
This difficulty leads \cite{DeGAmb89} to the space
$GSBV$ of \emph{generalized} special functions of bounded variation, i.e. the space of all 
real-valued maps $u$ whose truncates $(-N)\lor u\land N$ are all $SBV$. Since both the approximate gradient $\nabla u$
and the approximate discontinuity set $S_u$ behave well under truncation, it turns out that also the
energy of $u_n^N:=(-N)\lor u_n\land N$ is uniformly controlled, and now also $|Du_n^N|$; this is the very first
step in the proof of the compactness-lower semicontinuity theorem in $GSBV$, which shows that 
the sequence $(u_h)$ has limit points with respect to local convergence in measure, that any
limit point $u$ belongs to $GSBV$, and that 
$$
\int (|u|^s+|\nabla u|^p) d\Leb{m}\leq\liminf_h\int (|u_h|^s+|\nabla u_h|^p) d\Leb{m},\qquad
\Haus{m-1}(S_u)\leq\liminf_h\Haus{m-1}(S_{u_h}).
$$

In the higher codimension case, if we look for energies of the form
$$
\int (|u|^s+|\nabla u|^p+|M(\nabla u)|^\gamma) d\Leb{m}+\Haus{m-n}(S_{u})
$$
(with $s^{-1}+(n-1)/p<1$, $\gamma>1$) now involving also the minors $M(\nabla u)$ of
$\nabla u$, the same difficulty exists, but the truncation argument does not work anymore. Indeed,
the absence of $S_u$, namely the absolute continuity of $Ju$, may be due to very precise cancellation
effects that tend to be destroyed by a left composition, thus causing the appearance of new singular
points (see Example~\ref{ex:monopole} and the subsequent observation). Also, unlike the codimension 1 theory, no ``pointwise'' description
of $S_u$ is presently available.

For these reasons, when looking for compactness properties in $SB_nV$, we have been led to
define the space $GSB_nV$ of \emph{generalized} special functions of bounded higher variation
as the space of functions $u$ such that $Ju$ is representable in the form $R+T$, with
$R$ absolutely continuous with respect to $\Leb{m}$ and $T$ having finite \emph{size}
in an appropriate sense, made rigorous by the slicing theory of flat currents
(in the same vein, one can also define $GB_nV$, but our main object of investigation will be
$GSB_nV$). In particular, for $u\in GSB_nV$ the distribution $Ju$ is not necessarily 
representable by a measure. The similarity between $GSBV(\Omega)$ and $GSB_nV(\Omega)$ 
is not coincidental, and in fact we prove that in the scalar case $n=1$ these two spaces are essentially 
the same; on the contrary for $n\geq 2$ their properties are substantially different. 
In order to study the $T$ part of $Ju$ we use the notion of 
size of flat current with possibly infinite mass developed, even in metric spaces, in \cite{AmbGhi12},
see also \cite{HarDeP12} for the case of currents with finite mass.

The paper is organized as follows: after posing 
the proper definitions in the context of metric currents
 we briefly review the space $B_nV$ studied in 
 \cite{JerSon02,Del00,Del01}. In section~\ref{s:size}
 we present the notion of size of a flat current,  
 we relate it to the concept of distributional 
 jacobian and define our new space of functions $\GSB{m}{n}$. 
 The main result of the paper is presented 
 in section \ref{s:compactness}, where 
 with the help of the slicing theorem we will 
 generalize to our setting the compactness theorem 
 of $GSBV$, as well as the closure theorem
 in $SB_nV$ due to De Lellis in \cite{Del00,Del01}. 
  
 We finally apply the compactness theorem to show the existence
 of minimizers for a general class of energies that feature both a volume
  and a size term. The model problem is a new functional of Mumford-Shah
   type that we here introduce, in the spirit of \cite{DeGCarLea89}. 
We analyse its minimization together with 
suitable Dirichlet boundary conditions, both in the interior and in the 
closure of $\Omega$. In particular we show that minimizers 
must be nontrivial (i.e.: $S_u\neq\emptyset$), at least for suitable
boundary data; we also compare our choice of the energy
 with the classical $p$-energy of sphere-valued maps, see 
 \cite{GiaModSou,HarLinWan98,BreCorLie86}. 
 Regarding the problem in $\overline{\Omega}$ 
  the higher codimension of the singular set allows 
 concentration of the jacobian at the boundary, 
providing some interesting examples that we briefly include
in subsection \ref{ss:traces}. 
Similar variational problems in the framework 
of cartesian currents have been considered in \cite{Muc10}, 
where the author proves existence of minimizers in the set of 
maps whose graph is a normal current:  
the boundaries of these graphs enjoy a decomposition 
into vertical parts of integer dimension, inherited from general properties
 of integral currents, which relates to the space $B_nV$, see \cite{Muc11}.

In a forthcoming paper \cite{Ghi12} we show how the Mumford-Shah energy 
can be approximated, in the sense of $\Gamma$-convergence, by
 a family of functionals defined on maps with absolutely continuous
 jacobian: 
$$
F_\e(u,v;\Omega):= \int_\Omega (|u|^s+|\nabla u|^p+(v + k_\e)|M(\nabla u)|^\gamma) d\Leb{m}+\beta\int_\Omega \e^{q-n}|\nabla v|^q + \frac{(1-v)^n}{\e^n}\,d\Leb{m},
$$ 
(here $\beta, \gamma, q, k_\e$ are suitable parameters, $v:\Omega\rightarrow [0,1]$ is Borel). 
 Following \cite{AmbTor90,AmbTor92,AlbBalOrl05,ModMor77} 
 the control variable $v$ dims the concentration of $M(\nabla u)$: 
 the price of the transition between $0$ and $1$ is captured 
 by the Modica-Mortola term which detects $(m-n)$-dimensional sets.

We occasionally appeal to the metric theory of currents because 
the main tool in the definition of size and in the proof of 
the rectifiability theorem is the slicing technique, 
a basic ingredient of the metric theory. For instance the argument 
in \cite{AmbGhi12} proving the rectifiability
  of currents with finite size uses Lipschitz restrictions and maps of 
  metric bounded variation taking values into an appropriate space
 of flat chains with a suitable hybrid metric. However, no significant 
 simplification comes from the Euclidean theory, 
 except the fact that suffices to consider linear instead of
 Lipschitz maps. 
 
\subsection{Acknowledgements}
The authors wish to thank Camillo De Lellis, Nicola Fusco, 
Bernd Kirchheim, Domenico Mucci and Emanuele Spadaro 
for many useful discussions and comments. 	
This work has been supported by the ERC grant ADG GeMeThNES.

\section{Distributional Jacobians}\label{s:jac}
\noindent
We begin by fixing some basic notions about currents 
and recalling some properties of the distributional jacobian. 

\subsection{Exterior algebra and projections}
Our ambient space will be $\R^m$ with the standard basis
$e_1,\dots,e_m$ and its dual $e^1,\dots,e^m$. 
For every $1\leq k \leq m$ we let
$$
\Or_{k}=\big\{\pi:\R^m\rightarrow\R^{k}:
\pi\circ\pi^* = I_k\big\}
$$
be the space of orthogonal projections of rank $k$. 
We will also need to fix coordinates according to some 
projection $\pi\in\Or_{k}$: we agree that 
$\R^m\ni z = (x,y)\in\R^{k}\times\R^{m-k}$ are orthogonal coordinates
 with positive orientation such that $\pi(z)=x$. 
In particular we let $A^x = A\cap\pi^{-1}(x)$ be the 
 restriction of any $A\subset\R^m$ to the fiber $\pi^{-1}(x)$
 and $i^x=\R^{m-k}\rightarrow\R^m$ be the isometric 
 injection $i^x(y) = (x,y)$. 
 
As customary the symbols $\Lambda_k \R^m$ 
and $\Lambda^k \R^m$ will respectively denote the spaces of 
$k$-vectors and $k$-covectors in $\R^m$.  
The contraction operation 
$\res:\Lambda_{q}\R^m\times\Lambda^{p}\R^m\rightarrow \Lambda_{q-p}\R^m$ 
between a $q$-vector $\zeta$ and a $p$-covector $\alpha$, with $q\geq p$, is defined as:
\begin{equation}\label{restr}
\langle \zeta\res\alpha,\beta \rangle = \langle \zeta,\alpha\wedge\beta\rangle
\qquad\text{ whenever }
\beta\in\Lambda^{q-p}\R^m.
\end{equation}

If $L:\R^m\rightarrow\R^n$ is linear then 
\begin{equation}\label{M_nL}
M_nL:=e_1\wedge\dots\wedge e_m\res L^1\wedge\dots\wedge L^n \in\Lambda_{m-n}\R^m
\end{equation}
represents the collection of determinants of $n\times n$ 
minors of $L$.
In fact, if $\underline{i}:\{1,\dots,m-n\}\rightarrow \{1,\dots,m\}$
is an increasing selection of indexes, and if 
$\overline{i}:\{1,\dots,n\}\rightarrow \{1,\dots,m\}$ 
is the complementary increasing selection, 
then the $e_{\underline{i}}$ component of $M_nL$ is 
\begin{equation}\label{M_nLi}
\langle M_n L,e^{\underline{i}}\rangle = 
\langle e_1\wedge\dots\wedge e_m, L^1\wedge\dots\wedge L^n\wedge e^{\underline{i}}\rangle 
= (-1)^{\sigma}\det([L]_{\overline{i}}),
\end{equation}
where $[L]_{\overline{i}}$ is the $n\times n$ submatrix 
$L^{\ell}_j$ with $j=\overline{i}(1),\dots,\overline{i}(n)$, 
$\ell=1,\dots,n$ and $\sigma$ is the sign of the permutation 
$$
(1,\dots,m)\mapsto (\overline{i}(1),\dots,\overline{i}(n),\underline{i}(1),\dots,\underline{i}(m-n)).
$$
When $\rk(L)=n$, choosing an orthonormal frame $(e_i)$ so 
that $\ker(L)=<e_{n+1},\dots,e_m>$ we have $L = (A,{\bf 0})$ and
 by \eqref{M_nLi} $M_nL = \det(A)e_{n+1}\wedge\dots\wedge e_m$.
 In particular $M_nL$ is a simple $(m-n)$-vector.

Recall that the spaces $\Lambda_k\R^m$ and $\Lambda^k\R^m$ 
can be endowed with two different pairs of dual norms.
The first one is called norm, it is denoted by $|\cdot|$ and 
it comes from the scalar product where the multivectors 
\begin{equation}\label{frame}
\{e_{i(1)}\wedge\dots\wedge e_{i(k)}\}_{i} \quad\text{ and }
\quad \{e^{i(1)}\wedge\dots\wedge e^{i(k)}\}_{i}
\end{equation}
indexed by increasing maps $i:\{1,\dots,k\}\rightarrow \{1,\dots,m\}$ 
form a pair of dual orthonormal bases. The second one is 
called mass, comass for the space of covectors, and it
is defined as follows: the comass of $\phi\in\Lambda^k\R^m$ is
$$
\|\phi\| := \sup \big\{ \langle \phi,v_1\wedge\dots\wedge v_k\rangle : v_i\in \R^m, |v_i|\leq 1	\big\};
$$
and the mass of $\xi\in\Lambda_k\R^m$ is defined, by duality, by
$$
\|\xi\|:=\sup\big\{\langle \xi,\phi\rangle : \|\phi\|\leq 1\big\}.
$$
As described in \cite[1.8.1]{Fed}, in general $\|\xi\|\leq|\xi|$
and equality holds if and only if $\xi$ is simple. 

Therefore $|M_nL| = \|M_nL\|$. Moreover using the Pitagora's Theorem
for the norm and Binet's formula we have the following relation:
\begin{align}\label{normsup}
\sup_{\pi\in\Or_{m-n}}|M_nL\res d\pi| & =  \sup_{\pi\in\Or_{m-n}}|d\pi(M_nL)| =  
\sup_{\pi\in\Or_{m-n}}\big|\sum_{\underline{i}}M_nL^{\underline{i}}\,d\pi(e_{\underline{i}})\big| \notag\\
& \leq \sup_{\pi\in\Or_{m-n}}|M_nL|\big(\sum_{\underline{i}}|d\pi(e_{\underline{i}})|^2\big)^{\frac 12} = |M_nL|\sup_{\pi\in\Or_{m-n}}|\det(\pi\circ\pi^*)| = |M_nL|
\end{align}
where $d\pi$ stands for $d\pi^1\wedge\dots\wedge d\pi^{m-n}$, and 
the equality is realized by the orthogonal projection onto $\ker(L)$. 

We adopt the convention of choosing the mass and comass norms
 to measure the length of $k$-vectors and $k$-covectors respectively. 

\subsection{Currents in $\Omega\subset\R^m$}
We briefly recall the basic definitions and properties 
of classical currents in $\R^m$. This theory was introduced
 by De Rham in \cite{DeR}, along the lines of the previous
 work on distributions by Schwartz \cite{Sch}, and the subsequently
  put forward by Federer and Fleming in \cite{FedFle60}; 
  we refer to \cite{Fed} for a complete account of it. 
  The classical framework is best for treating the concepts of 
  distributional jacobian in a subset of $\R^m$; however we 
  will need to use the metric theory of Ambrosio and Kirchheim
   \cite{AmbKir00} to define the concept of size and of concentration
 measure. We will clearly outline the interplay between the two 
 approaches.
 
We give for granted the concepts of derivative, exterior differentiation, pull-back 
 and support of a test functions: they can all be defined 
 by expressing the form in the coordinates given by the 
 frame \eqref{frame}, see \cite[4.1.6]{Fed}. 
 
We begin by defining the space of smooth, compactly supported test forms:
\begin{definition}[Smooth test forms]\label{d:smoothforms}
We let $\DD^k(\Omega)$ be the space of smooth, compactly supported 
$k$-differential forms:
\begin{equation}
\DD^k(\Omega)= \bigcup_{K\Subset\Omega}\DD_K^k(\Omega), 
\qquad \DD^k_K(\Omega)=\big\{\omega\in 
C^{\infty}(\Omega,\Lambda^k\R^m),\,\spt(\omega)\subset K\big\}.
\end{equation}
Each space $\DD_K^k(\Omega)$ is endowed with the topology
given by the seminorms 
$$p_{K,j}(\omega)=\sup\{\|D^{\alpha}\omega(x)\|, x\in K,
|\alpha|\leq j\},$$ and $\DD^k(\Omega)$ is endowed with the finest topology
making the inclusions $\DD^k_K(\Omega)\hookrightarrow\DD^k(\Omega)$ 
are continuous.
\end{definition}
This topology is locally convex, translation invariant and Hausdorff;
moreover a sequence $\omega_j\rightarrow\omega$ in $\DD^k(\Omega)$ if 
and only if there exists $K\Subset\Omega$ such that $\spt(\omega_j)\subset K$ and
$p_{K,j}(\omega_j-\omega)\rightarrow 0$ for every $j\geq 0$.

\begin{definition}[Classical currents and weak* convergence]\label{d:current}
A current $T$ is a continuous linear functional on $\DD^k(\Omega)$. The space
 of $k$-currents is denoted by $\DD_k(\Omega)$. 
We say that a sequence $(T_h)$ weak* converges to $T$, $T_h\wstar T$,
whenever
\begin{equation}\label{wstar}
T_h(\omega)\rightarrow T(\omega)\qquad \forall\omega\in\DD^k(\Omega).
\end{equation}
\end{definition}
The support of a current is
$$
\spt(T):= \bigcap\{C: T(\omega)=0\,\,\forall\omega\in\DD^k(\Omega),\,\spt(\omega)\cap C = \emptyset\}.
$$
The boundary operator is the adjoint of exterior differentiation: 
$$
\partial T(\omega):=T(d \omega);
$$
let also $\phi:\Omega\rightarrow\R^{p}$ 
be a proper Lipschitz map: we let the push-forward of $T\in\DD_k(\Omega)$
 via $\Phi$ by duality:
 $$
 (\Phi_{\#}T)(\omega) := T(\Phi^{\#}\omega)\qquad \forall \omega\in\DD^{k}(\R^p).
 $$
According to \eqref{restr} given $T\in\DD_k(\Omega)$ and 
$\tau\in\DD^{\ell}(\Omega)$ with $\ell\leq k$ we set the restriction
$$
(T\res\tau)(\eta) = T(\tau\wedge\eta) \qquad \forall \eta\in\DD^{k-\ell}(\Omega).
$$

\begin{definition}[Finite mass and Normal currents]\label{d:mass}
 We say that $T\in\DD_k(\Omega)$ is a current of finite 
 mass if there exists a finite Borel measure $\mu$ in $\Omega$ 
 such that
\begin{equation}\label{mass}
|T(\omega)|\leq \int_{\Omega}\|\omega(x)\|d\mu(x)\qquad \forall \omega\in\DD^k(\Omega).
\end{equation}
The total variation of $T$ is the minimal $\mu$ satisfying 
\eqref{mass} and is denoted by $\|T\|$, and the mass
 $\M(T):=\|T\|(\Omega)$. 
As customary we let $\M_k(\Omega)$ be the space of 
finite mass $k$-dimensional currents
 and $\N_k(\Omega)$ be the subspace of normal currents:
 $$
 \N_k(\Omega)=\M_k(\Omega)\cap\{T:\partial T\in\M_{k-1}(\Omega)\}.
 $$
\end{definition}
Therefore every finite mass current can be represented as 
$T= \overset{\rightarrow}{T}\wedge \|T\|$, and admits 
and extension to $k$-forms with bounded Borel coefficients.
In particular we can restrict every finite mass current $T$ to
every open set $A$.
We denote 
$$
\E^m = e_1\wedge\dots\wedge e_m\wedge\Leb{m}
$$
the top dimensional $m$-current representing the Lebesgue 
integration on $\R^m$ with the standard orientation.
As explained in \cite[4.1.7, 4.1.18]{Fed}, \cite[3.2]{AmbKir00}, 
every function $f\in L^1(\Omega,\Lambda_{m-k}\R^m)$ 
induces a $k$-current of finite mass via the action 
\begin{equation}\label{Eresf}
(\E^m\res f)(\omega)=\int_{\Omega} \langle f\wedge\omega, e_1\wedge\dots\wedge e_m\rangle \,d\Leb{m} \qquad \forall \omega\in\DD^{k}(\Omega).
\end{equation}
Note that $f_h\w f$ weakly in $L^1$ entails $\E^m\res f_h\wstar \E^m\res f$.

\subsection{Flat currents}
In order to treat objects with possibly infinite mass, the 
right subspace of $\DD^k(\Omega)$ retaining some useful properties
such as slicing and restrictions is the space of flat currents. 
\begin{definition}[Flat norm and flat currents, {\cite[4.1.12]{Fed}}]\label{d:flat}
For every $\omega\in\DD^k(\Omega)$ we let 
$$
\F(\omega)=\max\big\{\sup_{x\in\Omega}\|\omega(x)\|,\sup_{x\in\Omega}\|d\omega(x)\|\big\}.
$$
The flat norm is defined as
\begin{align}
 \F(T) &= \inf\{\M(T - \partial Y) + \M(Y):\, Y\in \M_{k+1}(\Omega)\}\label{flatinf}\\
 &=\sup\{T(\omega):\,\omega\in\DD^k(\Omega),\, 
 \F(\omega)\leq 1\}. \label{flatsup}
\end{align}
The space  $\F_k(\Omega)$ of flat $k$-dimensional 
currents in an open subset $\Omega \subset\R^m$ 
is the $\F$-completion of $\N_k(\Omega)$ (see \cite{Fle66}
 and \cite[4.1.12]{Fed} for the equivalence of \eqref{flatinf} and \eqref{flatsup}).
\end{definition}
It is straightforward to prove that $\F$ is a norm; 
furthermore if $T$ is flat, so is $\partial T$ and 
$$
\F(\partial T)\leq \F(T) \leq \M(T).
$$

Throughout all the paper we will deal with three notions of convergence: 
\begin{itemize}
\item the convergence w.r.t. the flat norm $\F$ defined in \eqref{flatinf} above;
\item the weak convergence in $L^p$, $1\leq p <\infty$, denoted
 by $\w$;
\item the weak* convergence of currents \eqref{wstar}.
\end{itemize}
The map \eqref{Eresf} $f\mapsto \E^m\res f$ embeds $L^1(\Omega,\Lambda_{m-k}\R^m)$ 
into $\F_k(\Omega)$ by \cite[4.1.18]{Fed}, and the three aforementioned
topologies are ordered from the strongest to the weakest.

\subsection{Slicing}
As explained in \cite[4.2]{Fed} and \cite{AmbGhi12}, every 
$T\in\F_k(\Omega)$ can be sliced via a Lipschitz map 
$\pi\in\Lip(\Omega,\R^{\ell})$, $1\leq\ell\leq k$:
the result is a collection of currents
$$
\langle T,\pi,x\rangle \in\F_{k-\ell}(\Omega) \qquad 
\mbox{uniquely determined up to $\Leb{\ell}$ negligible sets}
$$
expressing the action of $T$ against tensor product
 forms $(\phi\circ\pi) d\pi\wedge\psi $, for 
 $\phi\in\DD^{0}(\R^{\ell})$ and $\psi\in\DD^{k-\ell}(\Omega)$:
$$ 
 T\big( (\phi\circ\pi) d\pi \wedge\psi\big) = 
 \int_{\R^{\ell}} \phi(x) \langle T,\pi,x\rangle(\psi)d\Leb{\ell}(x).
$$
The slices satisfy several properties: amongst them we recall 
\begin{equation}\label{slicecontroimm}
\text{$\langle T,\pi,x\rangle$ is concentrated on $\pi^{-1}(x)$ for
$\Leb{\ell}$-a.e. $x\in\pi(\Omega)$,}
\end{equation}
\begin{equation}
 \int_{\pi(\Omega)}\F(\langle T,\pi,x\rangle)\,d\Leb{\ell}(x)  \leq \Lip(\pi)^{\ell}\F(T), \label{fubinislices}
\end{equation}
and we refer to \cite[4.2.1]{Fed} and to \cite{AmbKir00} for a general account
in the Euclidean and general metric settings. 
We stress the following fact, which is a key tool to extend
many properties like restrictions and slicing from normal 
to flat currents, and that will be used
later on. Suppose $(T_h)\subset\F_k(\Omega)$ satisfy
 $$
 \sum_h \F(T_{h+1}-T_h)<+\infty
 $$
and let $\pi\in\Lip(\Omega,\R^{\ell})$ fixed. 
Then 
$$
\F\big(\langle T_h,\pi,x\rangle - \langle T,\pi,x\rangle\big)\rightarrow 0
$$
for $\Leb{\ell}$-almost every $x\in\R^{\ell}$. 
Recall that for the finite mass current $R = \rho\Leb{m}$ with
$\rho\in L^1(\Omega,\Lambda_{k}\R^m)$ Federer's coarea formula implies that 
at almost every $x\in\R^{\ell}$ it holds:
\begin{equation}\label{sliceac}
\langle R,\pi,x\rangle = (\rho(x,\cdot)\res d\pi)\,\Haus{m-\ell} \res\pi^{-1}(x).
\end{equation}

\subsection{Distributional jacobian}
We will assume throughout all 
the paper that $m\geq n$ are positive integers 
and that $p$ and $s$ are positive exponents satisfying
\begin{equation}\label{exp}
\quad \frac{1}{s}+\frac{n-1}{p} \leq 1.
\end{equation} 
The definition of distributional jacobian takes
 advantage of the divergence structure of jacobians
 $$
 d(u^1 du^2\wedge\dots\wedge du^n) = du^1 \wedge\dots\wedge du^n \quad \forall u\in C^1(\R^n,\R^n),
 $$
which allows to pass the exterior derivative to the test 
form and hence weakens the minimal regularity assumptions 
on the map $u$.

\begin{definition}[Distributional Jacobian]\label{d:djac} 
Let $u\in  W^{1,p} (\Omega, \R^n) \cap L^{s}(\Omega, \R^n)$. 
We denote by $j(u)$ the $(m-n+1)$-dimensional flat current
\begin{equation}\label{prejac}
 \langle j(u),\omega \rangle :=(-1)^n \int_{\Omega} u^1 du^2\wedge\dots\wedge du^n\wedge \omega\qquad
  \forall \omega \in \DD^{m-n+1}(\Omega);
  \end{equation}
we define the distributional Jacobian of $u$ as the $(m-n)$-dimensional 
flat current
\begin{equation}\label{jac}
Ju := \partial j(u)\in\F_{m-n}(\Omega). 
\end{equation}
\end{definition}

A few observations are in order: first of all the integrability 
assumption $u\in W^{1,p}\cap L^{s}$ and the exponent bound \eqref{exp}
ensure that $j(u)$ is a well-defined flat current of finite mass, 
since it acts on test forms as the 
integration against an $L^1(\Omega,\Lambda_{m-n+1}\R^m)$
function: $j(u) = (-1)^n\E^m\res u^1du^2\wedge\dots\wedge du^n$. 
As a consequence $Ju\in\F_{m-n}(\Omega)$ as 
declared in \eqref{jac}. Furthermore for $p\geq \frac{mn}{m+1}$ 
the constraint \eqref{exp} is satisfied with the Sobolev
 exponent $p^*$ in place of $s$, hence definition \ref{d:djac}
makes sense for $u\in W^{1,p}$ in this range 
of summability.

In \cite{BreNgu11} the authors showed that $Ju$ can
 be defined in the space $W^{1-\frac{1}{m},m}(\Omega)$, 
 which contains $L^s\cap W^{1,p}$ for every $s,p$ as in \eqref{exp}.
This extension exploits the trace space nature of $W^{1-\frac{1}{m},m}$, 
expressing $Ju$ as a boundary integral in $\R^m_+$.

Finally in the special situation $n=1$ the minimal requirement to
 give meaning to \eqref{prejac} is $u\in L^1(\Omega)$, and the Jacobian 
reduces to the distributional derivative $Ju = -\partial(\E^m\res u)$: 
\begin{equation}\label{scalar}
\big\langle Ju,\sum_i (-1)^{i-1}\omega_i \widehat{dx^i}	\big\rangle = 
-\sum_i\int_{\Omega}u\frac{\partial \omega_i}{\partial x^i}dx = 
\sum_i \langle D_i u,\omega_i\rangle .
\end{equation}

Regarding the convergence properties of these currents,
we note the following:
\begin{proposition}\label{p:conv}
Let 
 $u_h,u \in W^{1,p}(\Omega,\R^n)\cap L^{s}(\Omega,\R^n)$ 
 satisfy
\begin{itemize}
 \item $u_h\rightarrow u$  in $L^s(\Omega,\R^n)$,
 \item $\nabla u_h \rightharpoonup \nabla u$ weakly in $L^p(\Omega,\R^{n\times m})$.  
\end{itemize}
Then $\F(Ju_h -Ju)\rightarrow 0$. 
\end{proposition}
\begin{proof}
 Let us rewrite the difference $u_h^1 du_h^2\wedge\dots\wedge du_h^n - u^1 du^2\wedge\dots\wedge du^n$
 in the following way:
 \begin{multline*}
  u_h^1 du_h^2  \wedge\dots\wedge du_h^n - u^1 du^2\wedge\dots\wedge du^n =\\
 = (u_h^1 - u^1) du_h^2\wedge\dots\wedge du_h^n + 
 u^1\sum_{k=2}^n du_h^2\wedge du_h^{k-1}\wedge d(u^k_h - u^k) \wedge du^{k+1}\wedge\dots\wedge du^n. 
\end{multline*}
We can actually write each addendum in the last 
 summation as
\begin{equation*} 
-(u^k_h - u^k)du_h^2\wedge du_h^{k-1}\wedge du^1 \wedge du^{k+1}\wedge\dots\wedge du^n + d\zeta^k_h,
 \end{equation*}
where we set 
\begin{equation}\label{chain}
\zeta^k_h = (-1)^{k-2} u^1(u^k_h - u^k) du_h^2
\wedge\dots\wedge du_h^{k-1}\wedge du^{k+1}\wedge\dots\wedge du^n 
\in L^1(\Omega,\Lambda^{n-2}\R^m).
\end{equation}

Notice that we can always assume $s\geq p$, hence
$\zeta^k_h\in L^1$. 
To show \eqref{chain} it is sufficient to
approximate both $u$ and $u_h$ in the strong 
topology with regular functions and apply the 
Leibniz rule; the same approximation shows that
$d\zeta^k_h\in L^1$ and hence 
$\int_{\Omega} d\zeta^k_h\wedge d\omega =0$ for each 
 $\omega \in \DD^{m-n+1}(\Omega)$. By the calculations
 above we can estimate
 \begin{align*}
  |\langle Ju_h - Ju,\omega\rangle| & = 
  |\langle j(u_h) - j(u),d\omega\rangle|  \\
 &\leq \|d\omega\|_{L^{\infty}} 
 \sum_{k=1}^n \|u^k_h - u^k\|_{L^s}\|du^1_h\|_{L^p}\cdots
 \|du^{k-1}_h\|_{L^p}\|du^{k+1}_h\|_{L^p}\cdots\|du^n_h\|_{L^p}  \\
 & \leq C\,\F(\omega) \left(\sup_h \|\nabla u_h\|_{L^p}\right)^{n-1} \|u_h-u\|_{L^s}.
 \end{align*}
Taking the supremum on test functions $\omega$ with 
$\F(\omega) \leq 1$
we immediately obtain the asserted convergence.
\end{proof}
A natural question is the relation between the 
summability exponent $p$ and the regularity
 of the distribution $Ju$. There is a main 
difference between $p\geq n$ and $p<n$: if the 
gradient $\nabla u$ has a sufficiently high summability, 
then $Ju$ is an absolutely continuous measure.
In fact let $u_h=u*\rho_h$,
where $\rho_h$ is a standard approximation of the identity: 
 since $p\geq n$ the continuous embedding 
$W^{1,p} \hookrightarrow W^{1,n}_{{\rm loc}}$
implies that $u_h\rightarrow u$ both in 
$W^{1,p}\cap L^{s}$  and $W^{1,n}_{{\rm loc}}$. 
Taking a test form 
$\psi$ with compact support 
we can use Proposition \ref{p:conv} to pass to the 
limit in the integration by parts formula
$$
\langle Ju_h,\psi\rangle = 
(-1)^n \int_{\Omega} u_h^1 du_h^2\wedge\dots\wedge du_h^n\wedge d\psi =
 \int_{\Omega} du_h^1\wedge du_h^2\wedge\dots\wedge du_h^n\wedge \psi,
$$
yielding 
$Ju = \E^m \res du^1\wedge\dots\wedge du^n$. 

On the other hand when $p<n$ there are several examples 
of functions whose jacobian is not in $L^1$: for instance when $m=n$
the ``monopole'' function $u(x):= \frac{x}{|x|}$ satisfies 
$Ju = \Leb{n}(B_1) \llbracket 0\rrbracket$, where 
$\llbracket 0\rrbracket$ is the Dirac's mass in the origin.
More generally:
\begin{example}[Zero homogeneous functions, $m=n$, {\cite[3.2]{JerSon02}}]\label{ex:monopole}
Let $\gamma:S^{n-1}\rightarrow \R^n$ be smooth and let $u(x):=\gamma(\tfrac{x}{|x|})$.
Then 
\begin{equation}\label{curve}
Ju =Area(\gamma) \llbracket 0 \rrbracket
\end{equation}
where $Area(\gamma)$ is the signed area enclosed by $\gamma$.  
\end{example}
\begin{proof}
Outside the origin $u$ is smooth and takes values into 
the $(n-1)$-dimensional submanifold $\gamma(S^{n-1})$, 
hence $\spt(Ju)\subset\{0\}$. 
Set $t=|x|$ and $y=\tfrac{x}{|x|}$: 
then $d\gamma = \sum_i \frac{\partial u }{\partial x^k} tdy^k$, so
$$t^{n-1}du^2\wedge \dots \wedge du^n = d\gamma^2\wedge \dots\wedge d\gamma^n
\in \Lambda^{n-1}{\rm Tan\,} S^{n-1}.$$
Hence the only term of $d\omega$ surviving in the wedge product
 is $\tfrac{\partial \omega}{\partial t}dt$. Therefore
\begin{align}
 (-1)^n\int_{\R^n}u^1 &du^2\wedge\dots\wedge du^n\wedge d\omega = 
 (-1)^n\int_{\R^n}\frac{\partial \omega}{\partial t} \gamma^1(y) 
 du^2\wedge\dots\wedge du^n\wedge dt\notag\\
&= -\int_{\R^n}\frac{\partial \omega}{\partial t} u^1(y) dt\wedge du^2\wedge\dots\wedge du^n \notag\\
&= -\int_{\partial B_t}\left( \int_0^{+\infty} \frac{\partial \omega}{\partial t} dt \right)u^1(x) du^2\wedge\dots\wedge du^n \notag \\
&= -\int_{\partial B_1}\left( \int_0^{+\infty} \frac{\partial \omega}{\partial t} dt \right)\gamma^1(y) d\gamma^2\wedge\dots\wedge d\gamma^n \notag \\
& =\omega(0)  \int_{S^{n-1}}\gamma^1(y) d\gamma^2\wedge\dots\wedge d\gamma^n. \label{lastintegral}
\end{align}
Setting $\Upsilon(t,y):=t\gamma(y)$ the Lipschitz extension
 to the unit ball $B_1\subset \R^n$, by Stokes' Theorem \eqref{lastintegral} equals to 
\begin{align}
\omega(0)\int_{\partial B_1}\Upsilon^1(1,y) d\Upsilon^2\wedge\dots\wedge d\Upsilon^n 
&= \omega(0) \int_{B_1}d\Upsilon^1\wedge\dots\wedge d\Upsilon^n \notag\\
 &= \omega(0)\int_{B_1}\det(\nabla\Upsilon) dx
 =\omega(0) \int_{\R^n}\deg(\Upsilon,w,B_1)dw. \label{signedarea}
\end{align}
It is well known that \eqref{signedarea} represents the
signed area enclosed by the surface $\gamma(S^{n-1})$.
\end{proof}
This example immediately outlines one of the biggest differences
with the scalar case. 
Consider as in \cite{JerSon02} the ``eight-shaped" loop in $\R^2$ :
\begin{equation}\label{ottosemplice}
\gamma(\theta)=
\left\{
\begin{array}{ll}
& (\cos(2\theta)-1,\sin(2\theta))\qquad\mbox{ for }\theta\in[0,\pi],\\
&(1-\cos(2\theta),\sin(2\theta))\qquad\mbox{ for }\theta\in[\pi,2\pi].
\end{array}
\right.
\end{equation}
and let $u$ be the zero homogeneous extension. 
$\gamma$ encloses the union $B_1(-e_1)\cup B_1(e_1)$ with
degree $+1$ and $-1$ respectively: in light of \eqref{curve} $Ju=0$.
However a left composition with a smooth map $F:\R^2\rightarrow\R^2$
 easily destroys the cancellation, causing the appearance
  of a Dirac's mass in $0$. Hence the estimate 
\begin{equation}\label{Hadamard}
\|J(F\circ u)\|\leq \Lip(F)^2\|Ju\|
\end{equation} 
doesn't hold anymore if $u$ is not regular. Note that this 
phenomenon does not appear for $n=1$ and $u\in BV(\Omega)$, 
as Vol'pert chain rule provides exactly the estimate \eqref{Hadamard}
 (see \cite[Theorem 3.96]{AmbFusPal}). 
 
The failure of \eqref{Hadamard} is related to the validity
of a strong coarea formula for jacobians of vector valued maps, 
namely equation (1.7) in \cite{JerSon02}. 
The (weak) coarea amounts instead to decompose the current $Ju$ of
 a $B_nV$ map (see next paragraph for the definition) into the 
 superposition of integral currents corresponding to the level 
 sets of $u$: letting $u_{y}(x):=\tfrac{u(x)-y}{|u(x)-y|}$, it
 is proved in \cite[Theorem 1.2]{JerSon02} that 
$$
Ju=\frac{1}{\Leb{n}(B_1)}\int_{\R^n} Ju_y\,dy
$$
as currents. However, because of some cancellation phenomena
 like in \eqref{ottosemplice}, \eqref{Hadamard}, the strong version
 of the coarea formula
\begin{equation}\label{strongcoarea}
\|Ju\|=\frac{1}{\Leb{n}(B_1)}\int_{\R^n} \|Ju_y\|\,dy
\end{equation}
might well fail. Once again observe that for $n=1$ the
equality \eqref{strongcoarea} has been proved by Fleming and 
Rishel to holds for every $u\in BV$, see \cite[Theorem 3.40]{AmbFusPal}.
For a more detailed analysis we refer to
\cite{JerSon02,Del01,MulSpe95,DeP12}.

For later purposes we report the dipole construction, 
introduced by Brezis, Coron and Lieb in \cite{BreCorLie86}:
 it consists of a map taking values into a sphere which is constant
 outside a prescribed compact set, its jacobian is the difference of two Dirac's masses
  and satisfies suitable $W^{1,p}$ estimate. We write $(y,z)\in\R^{n-1}\times\R$ and denote by 
$\north = (0,1) \in S^{n-1}$ the north pole.
\begin{example}[Dipole, {\cite[2.2]{BreNgu11}}]\label{ex:dipole}
Let $n\geq 2$, $\nu\in \Z$, $\rho>0$: there exists a map $f_{\nu,\rho}:\R^n\rightarrow S^{n-1}$
with the following properties:
\begin{itemize}
\item $f_{\nu,\rho}\equiv\north$ outside $\{|y|+|z|<\rho\}$;
\item $f_{\nu,\rho} - \north \in W^{1,p}(\R^n,\R^n)$ for every $p<n$ with estimates
$$
\|\nabla f_{\nu,\rho}\|^p_{L^p} \leq C_p\, \nu^{\frac{p}{n-1}}\rho^{n-p};
$$
\item $Jf_{\nu,\rho} = \nu\Leb{n}(B^n_1)\big(\llbracket (0,-\rho) \rrbracket - \llbracket (0,\rho) \rrbracket\big)$.
\end{itemize}
\end{example}

The locality of the dipole construction allows 
  to glue several copies of dipoles to produce interesting examples. 

\begin{example}[Finiteness of $\F(Jg)$ does not imply finiteness of $M(Jg)$]\label{ex:dipole1}
We build a map $g$ such that $\F(Jg)$ is finite and $Jg$ has an infinite mass.
The construction starts from $f_{\nu,\rho}(\cdot,0):\R^{n-1}\rightarrow S^{n-1}$,
 which is a smooth map equal to $\north$ outside $B^{n-1}_{\rho}$ and such that
$\deg(f_{\nu,\rho}(\cdot,0)) = \nu$. For $|z|<\rho$ 
we extend by $f_{\nu,\rho}(y,z) = f_{\nu,\rho}(\tfrac{\rho y}{\rho- |z|},0)$ 
and we set $f_{\nu,\rho}\equiv\north$ at points $|z|\geq\rho$. 
Choosing a sequence of positive radii $(\rho_k)$ 
we can glue an infinite number of dipoles along the $z$ axis:
\begin{equation*}
g(y,z) = f_{1,\rho_k}(y, z -z_k)
 \qquad\mbox{ for }\quad |z - z_k|\leq 2\rho_k,
\end{equation*} 
where $z_0=0$ and $z_k =  2\sum_{j=0}^{k}\rho_j$.
The function $g$ belongs to $L^{\infty}\cap W^{1,p}$ provided $\sum_k\rho_k^{n-p}<\infty$:
 in this case note that
 $$
 Jg = \Leb{n}(B_1)\sum_k \llbracket (0,z_k-\rho_k)\rrbracket - \llbracket(0,z_k+\rho_k)\rrbracket,
 $$
hence $\F(Jg)\leq 2\Leb{n}(B_1)\sum_k\rho_k<\infty$ but $\M(Jg)=+\infty$.
\end{example}

More complicated examples, including maps such that
 $Ju$ is not even a Radon measure, are presented in 
 \cite{JerSon02, MulSpe95, AlbBalOrl05}.

\subsection{Functions of bounded $n$-variation}\label{ss:bnv}
The space of functions of bounded $n$-variation has 
been introduced by Jerrard and Soner in the fundamental 
paper \cite{JerSon02}.
\begin{definition}\label{d:B_nV}
$B_nV(\Omega,\R^n)$ is the space of functions 
$u\in W^{1,p}(\Omega,\R^n)\cap L^s(\Omega,\R^n)$ 
such that $Ju$ is a current of finite mass.
\end{definition} 
Notice that in this case the action of $Ju$ can be
 represented as the integration against a $\Lambda_{m-n}\R^m$-valued 
 Radon measure. Clearly the following statement is an easy 
 improvement of \ref{p:conv}, since every continuous function
  with compact support can be uniformly 
  approximated by a Lipschitz function
   with the same $L^{\infty}$ bound.
\begin{corollary}
Assume the same hypotheses of Proposition \ref{p:conv}. 
If in addition 
$$
(u_h)\subset B_nV(\Omega,\R^n) \qquad\text{and}\qquad \|Ju_h\|(\Omega) \leq C <\infty
$$ 
then 
$u\in B_nV(\Omega,\R^n)$ and $Ju_h\wstar Ju$ 
in the sense of measures.
\end{corollary}
Furthermore it is a general result 
 on normal currents, contained for example in \cite[4.1.21]{Fed} and
  in \cite[Theorem 3.9]{AmbKir00} for the metric spaces statement, 
  that if $T$ is a normal $k$-current then $\|T\|\ll\Haus{k}$.
In light of the Example \ref{ex:cantor} 
 (with a trivial extension in case $m>n$) $\|Ju\|\ll\Haus{m-n}$ 
 is the only possible bound on the Hausdorff dimension of 
$Ju$.

As in the theory of $BV$ functions $Ju$ satisfies a canonical
decomposition in three mutually singular parts according
to the dimensions (see \cite{Del00,AmbFusPal,JerSon02}): 
\begin{equation}\label{dec}
Ju = \nu \cdot \Leb{m} + J^c u + \theta\cdot \Haus{m-n}\res S_u 
\end{equation}
where the decomposition is uniquely determined by these three properties:
\begin{itemize}
\item $\nu =\frac{dJu}{d\Leb{m}} \in L^1(\Omega, \Lambda_{m-n}\R^m)$
is the Radon-Nikodym derivative of $Ju$ with respect to $\Leb{m}$;
\medskip
\item $\|J^cu\|(F)=0$ whenever $\Haus{m-n}(F)<\infty$;
\medskip
\item $\theta \in L^1(\Omega,\Lambda_{m-n}\R^m,\Haus{m-n})$ 
is a $\Haus{m-n}$-measurable function and $S_u\subset\Omega$ is $\sigma$-finite w.r.t. $\Haus{m-n}$. 
\end{itemize}
The intermediate measure $J^c u$ is known as the Cantor part of $Ju$.

\begin{example}[Summability exponent $p$ versus $\dim_{\mathscr{H}}\spt(Ju)$, 
{\cite[Theorem 5.1]{Mul93}}]\label{ex:cantor}
For every $\alpha\in [0,n]$ 
there exists a continuous $B_nV$ map 
$$
u_{\alpha}\in C^0(\R^n,\R^n)\cap\bigcap_{p<n}W^{1,p}_{\loc}(\R^n,\R^n)
$$ 
such that $Ju_{\alpha}$ is a nonnegative Cantor measure 
satisfying
$$
c\Haus{\alpha}\res \spt(Ju_{\alpha})  \leq Ju_{\alpha} \leq C \Haus{\alpha}\res \spt(Ju_{\alpha})
$$
for some $c,C>0$. In particular $\spt(Ju)$ has Hausdorff dimension $\alpha$. 
\end{example}
Hence no bound on $p$ 
is sufficient to constrain the singularity of $Ju$. 
Adding $m-n$ dummy variables to the domain the same
examples show $\alpha$ can range in the interval $[m-n,m]$
regardless how close $p$ is to $n$.

It has been proved in \cite{Mul90} and \cite{DelGhi10} that 
\begin{equation}\label{eq:ambr2}
\nu(x) = M_n\nabla u(x)= e_1\wedge\dots\wedge e_m \res du^1(x)\wedge\dots\wedge du^n(x) \in\Lambda_{m-n}\R^m
\end{equation}
at $\Leb{m}$-almost every point $x\in\Omega$. 
The set $S_u$ is unique up to $\Haus{m-n}$-negligible sets, 
and can be characterized by 
$$
S_u :=\left\{x\in\Omega : \limsup_{\rho\downarrow 0}\frac{\|Ju\|(B_{\rho}(x))}{\rho^{m-n}} >0\right\}.
$$
Moreover $S_u$ it has been shown in \cite{Del00} using
some general properties of normal and flat currents that $S_u$
is countably $\Haus{m-n}$-rectifiable and that for $\Haus{m-n}$-a.e.
 $x\in S_u$ the multivector $\theta(x)$ is simple and it orients
 the approximate tangent space $\Tan{m-n}(S_u,x)$. 
 
 \begin{definition}\label{d:SBnV}
We denote, in analogy with the $SBV$ theory,  by $SB_nV$ the set of $B_nV$ functions such that $J^c u = 0$. 
\end{definition}

The space $SB_nV$ enjoys a closure property proved in \cite{Del00}:
\begin{theorem}[Closure Theorem for $SB_nV$]\label{t:chiusura}
Let us consider $u,\,u_h\in B_nV(\Omega,\R^n)$ 
 and suppose that
\begin{itemize}
\item[(a)] 
 $u_h\rightarrow u \mbox{ strongly in }L^s(\Omega,\R^n)$ and 
$\nabla u_h \rightharpoonup \nabla u\mbox{ weakly in }L^p(\Omega,\R^{n\times m})$,
\item[(b)] if we write 
\begin{equation*} 
Ju_h = \nu_h \cdot \Leb{m} + \theta \cdot \Haus{m-n}\res S_{u_h}
\end{equation*}
then $|\nu_h|$ are equiintegrable in $\Omega$ and $\Haus{m-n}(S_{u_h}) \leq C <\infty$.
\end{itemize} 
Then $u\in SB_nV(\Omega, \R^n)$ and
\begin{equation*}
\nu_h \rightharpoonup \nu \mbox{ weakly in }L^1(\Omega, \Lambda_{m-n}\R^m),\quad 
\Haus{m-n}(S_{u}) \leq \liminf_h \Haus{m-n}(S_{u_h}).
\end{equation*} 
\end{theorem}

\subsection{Slicing Theorem}
We aim to apply the slicing operation to $Ju \in\F_{m-n}(\Omega)$
in the special case $\ell = m-n$, thus reducing ourselves to
 $0$-dimensional slices; moreover we want to relate these 
slices to the Jacobian of the restriction $J(u|_{\pi^{-1}(x)})$.
In \cite{Del00}, the author extended a classical result
on restriction of $BV$ functions (see \cite[Section 3.11]{AmbFusPal})
to Jacobians:
\begin{theorem}[Slicing]\label{t:slicing} 
Let $u\in W^{1,p}\cap L^s(\Omega,\R^n)$ be a function, 
and let $\pi\in \Or_{m-n}$.
Then for $\Leb{m-n}$-almost every $x\in\R^{m-n}$
\begin{equation}
\langle Ju, \pi, x\rangle = (-1)^{(m-n)n}i^x_{\#}(Ju^x),
\end{equation}
where $u^x = u\circ i^x$. Moreover $u\in B_nV(\Omega,\R^n)$ 
if and only if for every 
$\pi\in\Or_{m-n}$  
the following two conditions hold:
\begin{align*}
\mbox{{\upshape (i)}}& \qquad  u^x\in B_nV(\Omega^x,\R^n)\quad\mbox{ for }
\quad\Leb{m-n}\mbox{-almost every }x\in\R^{m-n}, \\
\mbox{{\upshape (ii)}}& \qquad  \int_{\pi(\Omega)} |Ju^x|(\Omega^x) \,d\Leb{m-n}(x) <\infty.
\end{align*}
In this case the Distributional Jacobian of the restriction $u^x$ is equal 
(up to sign) to the slice of $Ju$ at $x$:
\begin{equation}
\langle Ju, \pi, x\rangle = (-1)^{(m-n)n}i^x_{\#}(Ju^x),
\end{equation}
and this slicing property holds separately for the absolutely continuous part, 
the Cantor part and the Jump part of $Ju$, namely:
\begin{itemize}
\item $\langle J^a u, \pi, x\rangle = (-1)^{(m-n)n}i^x_{\#}(J^a u^x),$
\item $\langle J^c u, \pi, x\rangle = (-1)^{(m-n)n}i^x_{\#}(J^c u^x),$
\item $\langle J^s u, \pi, x\rangle = (-1)^{(m-n)n}i^x_{\#}(J^s u^x).$
\end{itemize}
\end{theorem}

\section{Size of a current and a new class of maps}\label{s:size}
\noindent
As anticipated in the abstract, we are interested in broadening the class $B_nV$ to include
vector valued maps satisfying a weaker control than the mass bound: this lack of control on 
$\M(Ju)$ already appears in Theorem \ref{t:chiusura} when we require \emph{a priori} the limit $u$ to be in $B_nV$. 
We relax our energy by considering a mixed control
of $Ju$, where we bound part of the current $Ju$
with its size. In general it is 
possible to define $\Sz(T)$
for every flat current $T\in\F_k(\Omega)$, even if $T$ has 
infinite mass: this size quantity was introduced in \cite{AmbGhi12},
borrowing some ideas already used by Hardt
and Rivi\`ere in \cite{HarRiv03}, Almgren \cite{Alm86}, 
Federer \cite{Fed86}, and
agrees with the classical notion of size for finite mass
currents. For example a polyhedral 
chain 
$$
P = \sum_{i=1}^n a_i \llbracket Q_i \rrbracket
$$
where $a_i\in\R$ and  $\llbracket Q_i \rrbracket$ are
 the integration currents over some
  pairwise disjoint $k$-polygons
  $Q_i$, has mass
  $\M(P) = \sum_i |a_i|\Haus{k}(Q_i)$ 
and size $\Sz(P) = \sum_i\Haus{k}(Q_i)$. 
The main idea behind the definition is to 
detect the support of the $0$-dimensional slices of $T$
via some $\pi\in\Or_k$ and then to optimize 
the choice of projection $\pi$. 
\begin{definition}[Size of a flat current, {\cite[Definition 3.1]{AmbGhi12}}]\label{d:size}
 We say that $T\in\F_k(\Omega)$ has finite size if
there exists a positive Borel measure $\mu$ such that 
\begin{align}\label{mu}
& \Haus{0}\res\spt(T) \leq \mu &\mbox{ in the case }k=0,\notag \\
& \mu_{T,\pi} := \int_{\R^k}\Haus{0}\res \spt\langle
T,\pi,x\rangle\,d\Leb{k}(x)\leq \mu \,\,\,\quad\forall \pi
\in\Or_k &\mbox{ in the case }k\geq 1.
\end{align}
The choice of $\mu$ can be optimized by choosing the least
upper bound of the family $\{\mu_{T,\pi}\}$ in the lattice of
nonnegative measures:
\begin{equation}\label{supmu}\mu_{T} := \bigvee_{\pi\in\Or_k}\mu_{T,\pi}= \bigvee_{\pi\in\Or_k}\int_{\R^k}\Haus{0}\res \spt\langle
T,\pi,x\rangle\,d\Leb{k}(x) .
\end{equation}
We set $\Sz(T):=\mu_T (\Omega)$.
\end{definition}

It can be proved (see \cite{AmbGhi12}) that for every flat 
current with finite size there exists 
unique (up to null sets) countably $\Haus{m-n}$-rectifiable set, 
denoted $set(T)$, such that
$$
\mu_T= \Haus{m-n}\res set(T),
$$
so that in particular $\Haus{m-n}(set(T)) = \Sz(T)$.
A pointwise constructions of $set(T)$ can also be given as follows
$$
set(T) :=\left\{x\in\Omega : \limsup_{\rho\downarrow 0}\frac{\mu(B_{\rho}(x))}{\rho^{m-n}} >0\right\}.
$$

The following result, which we will not use, holds for a fairly
 general class of metric spaces and fits
 naturally in the context of calculus of variations:
\begin{theorem}[Lower semicontinuity of size, {\cite[Theorem 3.4]{AmbGhi12}}]\label{tlsc}
Let $(T_h)\subset \F_k(\Omega)$ be a sequence of currents with
equibounded sizes and converging to $T$ in the flat norm:
$$
\Sz(T_h)\leq C<\infty,\qquad \lim_{h}\F(T_h-T)= 0.
$$
Then $T$ has finite size and
\begin{equation}
\Sz(T)\leq \liminf_h \Sz(T_h).
\end{equation}
\end{theorem}

We remark that the definition of size in the metric space
 contest of \cite{AmbGhi12} is slightly different, since supremum \eqref{supmu}
 was taken among all $1$-Lipschitz maps $\pi\in\Lip_1(\Omega,\R^{k})$. 
 However, when the ambient space is Euclidean, the rectifiability and lower semicontinuity results obtained there, 
 as well as the characterization of $\mu_T$ in terms 
  of $set(T)$ can be readily proved using only the subset of orthogonal projections.

The space of generalized functions of bounded higher variation 
 is described in terms of the decomposition \eqref{dec}: 
 we relax the requirement on the addendum of lower dimension
 and require only a size bound, retaining the mass bound 
 on the diffuse part.
Following the previous definitions 
we consider the Sobolev functions $u$ whose jacobian
can be split in the sum of two parts, $R$ and $T$, such that:
\begin{itemize}
 \item $R$ has finite mass and $\|R\|(F)=0$ whenever $\Haus{m-n}(F)<\infty$; 
 \medskip
\item $T$ is a flat chain of finite size.
\end{itemize}
In formulas:
\begin{definition}[Special functions of bounded higher variation]\label{d:GSBnV} 
The space of generalized functions of bounded higher variation is defined by
\begin{multline}
\GB{m}{n}(\Omega) =  \big\{ 
u\in W^{1,p}(\Omega,\R^n) \cap L^s(\Omega,\R^n): 
Ju = R + T,\, \M(R)+ \Sz(T)<\infty,\, \\ 
\|R\|(F)=0\, \forall F :\Haus{m-n}(F)<\infty
  \big\}.
\end{multline}
Analogously, the space of generalized special functions of bounded higher variation
is defined by
\begin{multline}\label{GSBnV}
\GSB{m}{n}(\Omega) =  \big\{ 
u\in W^{1,p}(\Omega,\R^n) \cap L^s(\Omega,\R^n): 
Ju = R + T,\, \M(R)+ \Sz(T)<\infty,\, 
\|R\|\ll\Leb{m}\big\}.
\end{multline}
In accordance with the classical $BV$ theory we denote $S_u := set(T_u)$. 
\end{definition}
This space is clearly meant to mimic the 
aforementioned $SB_nV$ class.
In particular, thanks to the slicing properties of flat currents and the definition of size, 
the slicing theorem for $\GSB{m}{m}({\Omega})$ can be stated in the following way: 
\begin{align}\label{slicingGSBnV}
u\in\GSB{m}{n}(\Omega)\Longleftrightarrow\forall\pi\in\Or_{m-n}\quad 
\left\{
\begin{array}{ll}
& u^x\in\GSB{n}{n}(\Omega^x),\\
& \\
 &\int_{\pi(\Omega)} \M(R_{u^x}) + \Sz(T_{u^x}) \,d\Leb{m-n}(x) <\infty.
\end{array}
\right.
\end{align}
In the following propositions we describe some 
useful properties of the class $\GSB{m}{n}(\Omega)$.
\begin{lemma}\label{l:sjbvn}
 If $m=n$ then $\GSB{n}{n}(\Omega) = SB_nV(\Omega)$.
\end{lemma}
\begin{proof}
 The statement relies on the fact that a flat $0$-current 
of finite size coincides with a finite sum of Dirac masses, 
and in particular it has finite mass. 
This property, reminiscent of  Schwartz lemma for 
distributions, has been proved in \cite{AmbGhi12}, Theorem 3.3.
Therefore the current 
$$
T = Ju - R  
$$
has finite mass, hence $\M(Ju) \leq \M(R) + \M(T) <\infty$ which
means $u\in B_nV(\Omega)$. 
\end{proof}

Since the Radon-Nikodym decomposition of a measure 
into the sum of an absolutely continuous and
a singular part is unique, by slicing also $R$ and $T$ are 
uniquely determined in the decomposition. Therefore
we can write $Ju = R_u + T_u$, so that $S_u$ is a well
 defined set.
 
 A very well known space of functions implemented in the 
calculus of variations is $GSBV$. The main idea behind this space, 
introduced in \cite{DeGAmb89} (see also \cite[Section 4.5]{AmbFusPal}), 
is to consider functions $u$ whose derivative $Du$ loses any kind
of local integrability, but nevertheless retains some
of the structure of $SBV$ functions. Setting $u^N:=(-N)\lor u\land N$
for every $N>0$ we define 
$$
GSBV(\Omega)=\{u:\Omega\rightarrow\R\,\mbox{ Borel}: 
\, u^N\in SBV(\Omega)\,\mbox{ for all integers }N>0\}.
$$  
The countable set of truncation given by $N\in\N$ is enough to 
provide the existence of an approximate differential $\nabla^*u$
 and of a countably $\Haus{m-1}$-rectifiable singular set $S^*_u$ such 
 that for every $N$
$$
\|Du^N\|\leq |\nabla^* u|\chi_{\{|u|\leq N\}}\Leb{m} + 2N\Haus{m-1}\res S^*_u.
$$
Moreover an analog of the slicing theorem for $BV$ function
 is available also in $GSBV$, see \cite[Proposition 4.35]{AmbFusPal}.
\begin{proposition}[Comparison between $GSB_1V$ and $GSBV$]\label{GSBV}
A function $u\in GSB_1V(\Omega)$ if and 
only if $u\in GSBV(\Omega)$, $u\in L^1(\Omega)$, 
$\nabla^*u\in L^1(\Omega,\R^n)$ and 
$\Haus{m-1}(S^*_u)<\infty$.
\end{proposition}
\begin{proof}
With abuse of notation, motivated by \eqref{scalar} 
we identify for scalar functions the action of $Ju$ on 
$\DD^{m-1}(\Omega)$ with the action of the distributional
 derivative $Du$ on $C^{\infty}_c(\Omega,\R^n)$
 (see the map ${\mathbf D}^{m-1}$ in \cite[1.5.2]{Fed}).
Consider first the case $m=1$. Let $u\in GSB_1V(\Omega)$:
  writing $R_u = \rho\Leb{1}$ and
 $T_u = \sum_{k=1}^{\Sz(T_u)}a_k\llbracket x_k\rrbracket$, 
 thanks to \eqref{scalar} we know that for $\omega\in\DD^0(\Omega)$ 
 $$
\langle D u,\omega\rangle =    
\int_{\Omega}\rho\omega\,dx + \sum_{k=1}^{\Sz(T_u)}a_k\omega(x_k).
 $$
This proves that $u\in SBV(\Omega)$, $S_u\subset set(T_u)$ and
$  u'(x)=\rho(x)$ almost everywhere. In particular for $N>0$ fixed
\begin{equation}\label{boh}
\|Du^N\|\leq |\rho|\Leb{1} + 2N\Haus{m-1}\res set(T_u).
\end{equation}
For $m\geq 2$ the slicing Theorem \ref{t:slicing} applied to a coordinate
 projection onto a hyperspace implies that 
almost every slice $u^x$ is in $\GSB_1(\Omega^x)$, hence for every
$N>0$ the estimate \eqref{boh} holds for $u^x$. Integrating back we have 
$\|Du^N\|(\Omega)<\infty$, hence $u\in GSBV(\Omega)$.

On the other if $u\in GSBV(\Omega)\cap L^1(\Omega)$ 
we know that $Du^N\wstar Du$ in the sense of distributions, and
 also in the flat norm, since the weak derivative is a distribution of 
 order $1$. Moreover $\nabla u^N\rightarrow \nabla^*u$ strongly in $L^1$, 
 hence also in the flat norm. Therefore the jump parts also converge
  to some flat $T_u$:
 $$
D^j u^N \overset{\F}{\rightarrow} T_u\in\F_{m-1}(\Omega).
 $$
 Recall that for $v\in BV$ the jump part of the derivative $Dv$ can 
 be expressed in terms of the approximate upper and lower limits $v_{\pm}$
 and of the approximate tangent $(m-1)$-vector $\tau$ in the following way:
\begin{equation}\label{eq:jump} 
D^j v = (v_+ - v_-)\tau\,\Haus{m-1}\res S_{v}.
\end{equation}
Hence if $m=1$ then $\Haus{0}(\spt T_u)\leq \lim_N\Haus{0}(S_{u^N})\leq \Haus{0}(S^*_u)$;
 in the general case can be achieved using the slicing Theorem \ref{t:slicing} 
 and Proposition \cite[4.35]{AmbFusPal}.
 \end{proof}
 
 \subsection{Some examples}
The following observation shows that when $n\geq 2$ 
it is hopeless to rely on truncation to get mass bounds 
for $Ju$.

\begin{example}[$L^{\infty}$ bound for $n\geq 2$]\label{ex:Linfinity}
For $n\geq 2$ let $\gamma_k:S^{n-1}\rightarrow S^{n-1}$ 
be a smooth map with degree $k$, and call $u_k$ its zero
homogeneous extension to $\R^n$. Then $\|u_k\|_{L^{\infty}}\leq 1$
 but by Example \ref{ex:monopole} $Ju = k\Leb{n}(B_1)\llbracket 0 \rrbracket$.
\end{example}
 On the contrary for $n=1$ and $u\in BV(\Omega)$ the approximate upper and
  lower limits $u_{\pm}$ of $u$ characterize the singular set: 
  $S_u=\{ x\in\Omega: u_+(x)>u_-(x)\}$. Equation 
  \eqref{eq:jump} implies that an
  $L^{\infty}$ bound on $u$ together with a size bound 
 $\Haus{m-1}(S_u)<\infty$ gives a mass bound on $Du$.
 
 We now adapt the construction in \ref{ex:dipole1}, building
 a map whose jacobian has infinite mass but finite size: 
\begin{example}[$\Sz(Ju)<\infty$ but $\M(Ju)=\infty$]\label{ex:finitesizejac}
Set $m=n+1$ and let us write $(x,y,z)$ the coordinates of $\R\times\R^{n-1}\times \R$.
Besides $\nu\in\Z$ and $\rho>0$ fix 
an extra parameter $R\geq\rho$. We extend the function
 $f_{\nu,\rho}(y,z)$ of Example \ref{ex:dipole} to $\R^{n+1}$ by
\begin{equation*}
h_{\nu,\rho,R}(x,y,z)=\left\{
\begin{array}{ll}
 f_{\nu,\rho}\Big(\frac{Ry}{R-|x|},\frac{Rz}{R-|x|}\Big) 
\qquad &\mbox{ for }\quad |x|<R,\\
 \north\qquad &\mbox{ for }\quad |x|\geq R.
\end{array}
\right. 
\end{equation*}
Clearly $h_{\nu,\rho,R}\neq\north$ in the set $\{|x|/R + |y|/\rho +|z|/\rho<1 \}$;
 by simmetry we can do the computations in $\{x<0\}$.
Let us estimate the partial derivatives:
 \begin{align*}
&\Big|\tfrac{\partial h_{\nu,\rho,R}}{\partial x}(x,y,z)\Big| \leq 
\tfrac{R(|y|+|z|)}{(R+x)^2}|\nabla f_{\nu,\rho}|\Big(\tfrac{Ry}{x+R},\tfrac{Rz}{x+R}\Big)
 \leq \tfrac{\rho }{x+R} |\nabla f_{\nu,\rho}|\Big(\tfrac{Ry}{x+R},\tfrac{Rz}{x+R}\Big),\\
& |\nabla_{y,z} h_{\nu,\rho,R}| \leq \tfrac{R}{x+R}|\nabla f_{\nu,\rho}|  \Big(\tfrac{Ry}{x+R},\tfrac{Rz}{x+R}\Big).
 \end{align*}
Since $\rho\leq R $ 
\begin{align}
\int_{\R^{n+1}}|\nabla h_{\nu,\rho,R}|^p\,dxdydz &\leq 
2(n+1)\int_{-R}^0  \int_{\{|y|/\rho +|z|/\rho<(x+R)/R\}}  \Big(\tfrac{R}{x+R}\Big)^p 
|\nabla f_{\nu,\rho}|^p  \Big(\tfrac{Ry}{x+R},\tfrac{Rz}{x+R}\Big)dydz \,dx  	\notag \\
&\leq 2(n+1)\int_{0}^R \Big(\tfrac{R}{x}\Big)^{p-n} \int_{\{|y| +|z|<\rho\}} 
|\nabla f_{\nu,\rho}(y,z)|^p\,dydz\,dx \notag\\
&\leq C_p\, \nu^{\frac{p}{n-1}}\rho^{n-p}R. \label{pnorm}
\end{align}
Moreover $Jh_{\nu,\rho,R}$ is the integral cycle $\nu\cdot \zeta_{\#}\llbracket [0,1]\rrbracket$,
where $\zeta:[0,1]\rightarrow \R^{n+1}$ is the following closed curve:
\begin{eqnarray*}
 \zeta(t) =\left\{
\begin{array}{ll}
( 4Rt-R,0,-4\rho t) \quad &\mbox{ for }\quad t\in [0,\tfrac 14],\\
(4Rt-R,0,4\rho t-2\rho) \quad &\mbox{ for }\quad t\in [\tfrac 14,\tfrac 12],\\
(3R - 4Rt,0,4\rho t-2\rho) \quad &\mbox{ for }\quad t\in [\tfrac 12,\tfrac 34],\\
(3R - 4Rt,0,4\rho-4\rho t) \quad &\mbox{ for }\quad t\in [\tfrac 34,1].
\end{array}
\right. 
\end{eqnarray*}
Since $\Lip(\zeta)\leq CR$ we have $\M(Jh_{\nu,\rho,R})\leq C \nu R $
 and $\Sz(Jh_{\nu,\rho,R}) \leq CR$.
  Like in \ref{ex:dipole} we glue infinite copies of $h_{\nu_k,\rho_k,R_k}$
along the $z$ axis and obtain a map $g$: the Sobolev 
norm of $g$ can be estimated by \eqref{pnorm}:
\begin{equation*}\label{eq:gnorm}
\|\nabla g\|_{L^p}^p \leq C \sum_k \nu_k^{\frac{p}{n-1}}\rho_k^{n-p}R_k
\end{equation*}
and
\begin{eqnarray*}
&\M(Jg) \leq C \sum_k \nu_kR_k,\\
&\Sz(Jg) \leq C \sum_k R_k.
\end{eqnarray*}
Choosing $\nu_k = k$, $R_k=\tfrac{1}{k^2}$ and $\rho_k=e^{-k}$ 
we obtain a $S^{n-1}$-valued $W^{1,p}$ function constant outside a compact
 set and whose Jacobian has infinite mass but finite size.
\end{example}

\subsection{$Ju$ and approximate differentiability} We now extend 
to $\GSB_n$ the pointwise characterization of the absolutely continuous
 part of $Ju$.
\begin{proposition}[$Det = det$ in the $\GSBnV$ class]\label{p:det}
Let $u\in \GSB{m}{n}(\Omega)$  
and write $Ju = R + T$ as in Definition \ref{d:GSBnV}. 
Let $\nabla u$ be the approximate differential
of $u$. 
Then 
\begin{equation}\label{z}
\frac{dR}{d\Leb{m}} = M_n\nabla u
\quad\text{$\Leb{m}$-almost everywhere in $\Omega$.} 
\end{equation}
\end{proposition}
\begin{proof}
For the ease of notation let $\nu:=\frac{dR}{d\Leb{m}}$. 
Fix a projection $\pi\in\Or_{m-n}$ and let us write 
the coordinates $z= (x,y)$ accordingly. 
For a fixed $x\in\R^{m-n}$ we note that the injection $i^x$ and 
the complementary projection $\pi^{\perp}\sbarretta _{\pi^{-1}(x)}$ are
one the inverse of the other.
Recall the slicing Theorem for general Sobolev functions gives
\begin{equation}\label{eq:ambr1}
\langle Ju, \pi, x\rangle = (-1)^{(m-n)n}i^x_{\#}(Ju^x).
\end{equation}
Taking Lemma \ref{l:sjbvn} into account,
for almost every $x\in \R^{m-n}$ it holds $u^x \in B_nV(\Omega^x)$
and  $\M(\langle R,\pi,x\rangle)) + \Sz(\langle T,\pi,x\rangle) <\infty$, 
hence 
\eqref{sliceac} gives
\begin{equation}\label{eq:slicexesima}
\langle Ju, \pi, x\rangle  = \nu(x,\cdot)\res d\pi \Haus{n}\res\pi^{-1}(x) + \langle T,\pi,x\rangle. 
\end{equation}
Pushing forward \eqref{eq:slicexesima} via $\pi^{\perp}$ by \eqref{eq:ambr1} it follows that
$(-1)^{(m-n)n}Ju^x = \nu(x,\cdot)\res d\pi \Leb{n} + \tilde{T}^x$, 
with $\tilde{T}^x= \pi^{\perp}_{\#}\langle T,\pi,x\rangle$. 
But the finiteness of the size of $\langle T,\pi,x\rangle$
implies that $\tilde{T}^x$ is a 
sum of $\Sz(\langle T,\pi,x\rangle)$ Dirac masses. 
In particular by equation \eqref{eq:ambr2} in the case $m=n$  
we know that 
$$
 (-1)^{(m-n)n}\nu(x,\cdot)\res d\pi = \det \nabla_y u(x,\cdot).
$$
Using \eqref{M_nLi} we obtain $\nu(x,\cdot)\res d\pi = M_n\nabla u (x,\cdot)\res d\pi$
for almost every $x$.
We recover the equality \eqref{z} by taking orthogonal
projections $\pi$ onto every $(m-n)$-dimensional coordinate 
subspace. 
\end{proof}

It will be useful to extend the result of Proposition~\ref{p:det} to  
 the lower order determinants: let $u\in\GSB{n}{n}(\Omega)$ 
 and $w\in\Lip(\Omega,\R^n)$. We denote by $\Gamma(u,w)$ 
 the sum of the jacobians of the functions obtained by 
 replacing at least one component of $u$ with the 
 respective component of $w$, but not all of them. 
More precisely for every $I\subset\{1,\dots,n\}$ 
such that $0<|I|<n$ we construct the function 
$u_I$ whose components are
\begin{eqnarray*}
 u_I^k=\left\{
\begin{array}{ll}
u^k \text{ if }k\not\in I,\\
w^k \text{ if }k \in I.
\end{array}
\right. 
\end{eqnarray*}
Then we let $\Gamma(u,w)=\sum_{0<|I|<n}Ju_I$.
By the multilinearity of jacobians, it is easy to check that if $u$ is Lipschitz the identity
\begin{equation}\label{eq:ambr3}
J(u+w) =Ju+\Gamma(u,w)+Jw
\end{equation}
holds pointwise $\Leb{m}$-a.e. in $\Omega$.

\begin{corollary}\label{c:detgeneral}
 Let $\Omega\subset\R^m$, $w\in\Lip(\Omega,\R^n)$ and $u\in\GSB{m}{n}(\Omega)$.
 Then, in the sense of distributions, it holds
 \begin{equation}\label{eq:ambr4}
 J(u+w) =Ju+ \Gamma(u,w)+Jw.
 \end{equation}
\end{corollary}
\begin{proof} The proof uses the following observation: if $u_h\rightarrow u$ in $L^s$
and $\nabla u_h\w \nabla u$ in $L^p$, then 
by Reshetnyak's Theorem and the inequality $p>n-1$ every minor of $\nabla u$ of order
$k<n$ is weakly continuous in $L^{\frac pk}$. 
It follows that $\Gamma(u_h,w)\to\Gamma(u,w)$, so that we can pass to the limit in \eqref{eq:ambr3}
to obtain \eqref{eq:ambr4}.
\end{proof}

\section{Compactness}\label{s:compactness}
\noindent
\begin{theorem}[Compactness for the class $\GSB{m}{n}$]\label{t:compactness}
Let $s>0,\,p>1$ be exponents with  $\tfrac 1s+ \tfrac{n-1}{p}\leq 1$ and let 
$\Psi:[0,\infty)\rightarrow [0,\infty)$ be a 
convex increasing function satisfying 
$\lim\limits_{t\rightarrow\infty}{\Psi(t)}/{t} = \infty$.\\
Let $(u_h)\subset \GSB{m}{n}(\Omega)$ be such that
$u_h\rightarrow u$ in $L^s(\Omega,\R^n)$ and 
$\nabla u_h\rightharpoonup \nabla u$
 weakly in $L^p(\Omega,\R^{n\times m})$.
 Suppose that $Ju_h = R_{u_h}+T_{u_h}$ fulfil 
\begin{equation}\label{bound}
K:= \sup_h\int_{\Omega}\Psi\left(\big|\frac{dR_{u_h}}{d\Leb{m}}\big|\right) d\Leb{m} + \Sz(T_{u_h}) <\infty.
\end{equation}
Then $u\in\GSB{m}{n}(\Omega)$ and, writing $Ju = R_{u}+ T_u$,   
\begin{align}
& \frac{dR_{u_h}}{d\Leb{m}}\rightharpoonup \frac{dR_{u}}{d\Leb{m}}
\qquad\text{weakly in }L^1(\Omega,\Lambda_{m-n}\R^m),\label{wregpart} \\
\medskip
& \Sz(T_u)\leq \liminf_{h}\Sz(T_{u_h})\label{lscsize}.
\end{align}
\end{theorem}

\begin{proof} Without loss of generality we can assume $\Psi$ to have 
at most a polynomial growth at infinity, for otherwise 
it is sufficient to take $\tilde{\Psi}(t) := \min\{\Psi(t), t^2\}$. 
In particular we will use the inequality 
\begin{equation}\label{Delta_2}
\Psi(2t) \leq C\Psi(t) \qquad\forall t>0
\end{equation}
(this inequality is known as $\Delta_2$ condition in the literature, see for 
 instance \cite[8.6]{AdaFou}).  
We shorten $T_h, R_h$ in place of $T_{u_h}$ and $R_{u_h}$ respectively and
denote by $\rho_h=M_n\nabla u_h$ the densities of $R_h$ with respect to $\Leb{m}$.
 We know from Proposition \ref{p:conv}
that $Ju_h\rightarrow Ju$ in the flat norm. Possibly extracting a subsequence
we can assume with no loss of generality that: 
\begin{itemize}
\item [(a)] the limit $\lim_h \,\Sz(T_h)$ exists,
\item [(b)] $\rho_h\rightharpoonup \rho$ 
weakly in $L^1(\Omega,\Lambda_{m-n}\R^m)$,
\item [(c)] $(u_h)$ rapidly converges to $u$ in $L^s$: 
as a consequence of Proposition \ref{p:conv} we have also
\begin{equation}\label{fastS}
\sum_h \F(Ju_{h}-Ju) < \infty. 
\end{equation}
\end{itemize}
Indeed, if we prove the result under these additional assumptions, then we can use the
weak compactness of $\rho_h$ in $L^1$ provided by
the Dunford-Pettis Theorem,
and the fact that any subsequence admits a further
subsequence satisfying (a), (b), (c) to obtain the general statement.

We shall let $R:= \rho\Leb{m}$ be the limit current:
since the flat and weak convergences in (b) and
(c) are stronger than the weak* convergence
for currents, putting them together we obtain a flat current
$T:=Ju-R$ such that
\begin{equation}\label{wconv}
T_h \wstar T.
\end{equation}
The proof is divided in three steps: we first address 
the special case $m=n$, then we use this case and the 
slicing Theorem \eqref{slicingGSBnV} to show the lower 
semicontinuity of size in the second step.  
The main difficulty is in the third step, where we prove  
\eqref{wregpart}, because weak convergence behaves badly 
under the slicing operation.
\\ {\bf Step 1: $m=n$.} 
We can apply a very particular case of Blaschke's compactness 
Theorem \cite[4.4.15]{AmbTil} to the sets $\spt(T_h)$, which
 have equibounded cardinality, to obtain a finite set $N\subset\overline{\Omega}$  
and a subsequence $(T_{h'})$ such that
$\spt(T_{h'})\longrightarrow N$ in the sense of Hausdorff convergence. 
By \eqref{wconv} we immediately obtain that $\spt(T)\subset N\cap\Omega$, hence
$\Sz(T)<\infty$ and $u\in\GSB{n}{n}(\Omega)$. In addition, since any point
in $\spt T$ is the limit of points in $\spt T_{h'}$ it follows that
\begin{equation*}
\Sz(T)\leq\liminf_{h'}\Sz(T_{h'})=
\lim_{h}\Sz(T_{h}).
\end{equation*}
Finally, since $Ju = R + T$ it must be $T=T_u$, which yields \eqref{lscsize}, 
and $R=R_u$, which together with (b) yields \eqref{wregpart}.\\ 
{\bf Step 2: $m\geq n$.}  Let us fix $A\subset\Omega$ open, $\pi\in\Or_{m-n}$ and $\e\in (0,1)$: 
the bound \eqref{bound}, \eqref{sliceac} and Fatou's lemma imply that
\begin{align}\label{liminf}
 +\infty>K & \geq \liminf_h\left\{ \mu_{T_h}(A) + 
 \e\int_{A}\Psi(|\rho_h|) d\Leb{m}\right\} \\
   & \geq \liminf_h\left\{ \mu_{T_h,\pi}(A) + 
   \e\int_{A}\Psi(|\rho_h\res d\pi |) d\Leb{m} \right\}  \label{eq:ambr5}  \\
   & = \int_{\R^{m-n}}\liminf_h \left[ \Haus{0}(A^x\cap \spt(\langle T_h,\pi,x\rangle)) 
     +\e \int_{A^x}\Psi(|\rho_h\res d\pi |)dy \right]dx \notag\\
   & =   \int_{\R^{m-n}}\liminf_h \left[ \Haus{0}(A^x\cap \spt( T_{u^x_h})) 
     +\e \int_{A^x}\Psi(|\rho_h^x|)dy \right]dx,\label{finalfatou}
     \end{align}
with $\rho^x_h:=M\nabla u_h(x,\cdot)\res d\pi$. By \eqref{liminf} we can choose for almost every $x\in\R^{m-n}$  
 a subsequence $h' = h'(x,A)$, possibly depending
  on $x$ and on the set $A$, realizing the finite 
  lower limit:
\begin{equation*}
 \liminf_h \, \Haus{0}(A^x\cap \spt( T_{u^x_h}))
     +\e \int_{A^x}\Psi(|\rho_h^x|)dy.
\end{equation*}
Recall that thanks to (c) $Ju^x_h\overset{\F}{\rightarrow}Ju^x$ for almost every $x$.
 We can therefore apply step 1 to the sequence $u^x_h\in\GSB{n}{n}(\Omega^x)$,
 which converges rapidly to $u^x$, to conclude that
 $u^x\in\GSB{n}{n}(\Omega^x)$ and that 
\begin{align}\label{fatouchain}
 \Haus{0}(A^x\cap \spt( T_{u^x}))& \leq\liminf_{h'}\Haus{0}(A^x\cap \spt( T_{u^x_{h'}})) \notag\\
  & \leq \liminf_{h'}\Haus{0}(A^x\cap \spt( T_{u^x_{h'}})) 
  +\e \int_{A^x}\Psi(|\rho_{h'}^x|) dy \notag\\
 & = \liminf_{h}\Haus{0}(A^x\cap \spt( T_{u^x_{h}})) 
 +\e \int_{A^x}\Psi(|\rho_h^x|)dy. 
\end{align}
Integrating in $x$ and applying \eqref{eq:ambr5} as well as the monotonicity
 of $\Psi$ we entail
\begin{equation}
 \mu_{T,\pi}(A)\leq \liminf_h \left\{ \mu_{T_h,\pi}(A) + 
 \e\int_{A}\Psi(|\rho_h|) d\Leb{m} \right\} =:\eta_\e(A).
\end{equation}
The map $A\mapsto \eta_\e(A)$ is a finitely superadditive set-function, 
with $\eta_\e(\Omega)\leq\liminf_h\Sz(T_{u_h})+K\e$. 
 Therefore if $B_1,\dots,B_N$ are pairwise disjoint 
 Borel sets and $K_i\subset B_i$ are compact, 
 we can find pairwise disjoint open sets $A_i$ 
 containing $K_i$ and apply the superadditivity to get
$$
\sum_{i=1}^N \mu_{T,\pi_i}(K_i)\leq \sum_{i=1}^N \eta_\e(A_i)\leq\eta_\e(\Omega).
$$
Since $K_i$ are arbitrary, the same inequality holds 
with $B_i$ in place of $K_i$; since also $B_i$, $\pi_i$ and
$N$ are arbitrary, it follows that $\mu_T$ is a finite 
Borel measure and $\mu_T(\Omega)\leq\eta_\e(\Omega)$.
Hence $u\in \GSB{n}{n}(\Omega)$ because
 $Ju = R + T$, $\Sz(T)<\infty$ and $R$ is an
  absolutely continuous measure. 
Letting $\e\downarrow 0$ we also prove \eqref{lscsize}. 
For later purposes we notice that we proved
\begin{equation}\label{eq:lscA}
\mu_{T_u}(A)\leq \liminf_h \mu_{T_{u_h}}(A).
\end{equation}
{\bf Step 3: proof of \eqref{wregpart}.} 
In order to prove \eqref{wregpart}, since the space $\Lambda_{m-n}\R^m$ 
is finite dimensional, we will prove that
\begin{equation}\label{wregpartpi}
\rho_h\res d\pi\rightharpoonup M_n\nabla u\res d\pi 
\qquad\text{weakly in }L^1(\Omega,\Lambda_{m-n}\R^m)\end{equation}
for every orthogonal projection $\pi$ onto a coordinate
subspace. 
We fix an open $A\subset\Omega$ and $a\in\R$.
 From now on $w:A\rightarrow\R^n$ will be an affine map such that 
$$
\nabla_x w=0,\qquad \det(\nabla_y w) = a.
$$
Let us compute $J(u_h+w)$:  
thanks to Corollary~\ref{c:detgeneral} we get
$$
 J(u_h+w)  =  Ju_h + \Gamma(u_h,w) + a\E^m\res d\pi.
$$
We are now ready to prove the last part of the Theorem. 
We argue as in step 2, but this time we change 
the form of the energy and we analyse the convergence
of a perturbed sequence of maps. 
First of all we note that the sequence
\begin{equation}\label{pertenergy}
\int_{A}\Psi \left( \big| \rho_h\res d\pi + a \big| \right) d\Leb{m}
 + \e\mu_{T_h,\pi}(A) + \e\int_A|\nabla u_h|^p d\Leb{m}
\end{equation}
is still bounded from above, because 
$|\alpha+\beta|^p\leq 2^{p-1}(|\alpha|^p+|\beta|^p)$, \eqref{Delta_2} and 
the convexity of $\Psi$ imply that 
\begin{equation*}\label{Deltabound}
\int_{A}\Psi \left( \big| \rho_h\res d\pi + a \big| \right) d\Leb{m} \leq
\frac C2 \int_{A}\Psi \left( \big| \rho_h \big| \right) d\Leb{m} 
+\frac C2 \Psi(|a|)\Leb{m}(A) \leq \frac C2 \big(K + \Psi(|a|)\Leb{m}(A)\big).
\end{equation*}
We consider the sequence $(u_h + w)\subset\GSB{m}{n}(A)$ and 
the perturbed energy \eqref{pertenergy}:
arguing as in the chain of inequalities \eqref{liminf}-\eqref{finalfatou} 
for almost every $x$ we can find a suitable 
 subsequence $h'=h'(x,A)$ realizing the finite lower limit of the sliced energies
\begin{equation}\label{eq:energyx}
\int_{A^x}\Psi \left( \big| \rho^x_h+ a \big| \right) dy
 + \e\Haus{0}(A^x\cap \spt( T_{u^x_h})) + \e\int_{A^x}|\nabla u^x_h|^p dy.
\end{equation}
Since $\Psi$ is superlinear at infinity, up to subsequences
 the densities $\rho^x_{h'} + a$ weakly converge to 
 some function $r^x$ in $L^1(A^x)$: in particular the associated
  currents weak* converge
\begin{equation}\label{eq:rx}
 (\rho^x_{h'} + a)\E^{n}\res A^x \wstar r^x \E^{n}\res A^x.
\end{equation}
 Thanks to the fast convergence (c) we also know that 
 $u^x\rightarrow u$ in $L^s(A^x)$; moreover the 
 boundedness of the Dirichlet term in \eqref{eq:energyx} implies also that 
 $\nabla_y u^x_{h'}\w \nabla u^x$ in $L^p(A^x,\R^n)$, 
hence by step 2 we get 
$$
u^x\in\GSB_n(A^x)\qquad\mbox{ and }\qquad T_{u^x_{h'}}\wstar T_{u^x}.
$$ 
The weak convergence of the gradients in $L^p$ also allows 
 to use the continuity property of $\Gamma(\cdot,w^x)$
 along the sequence of restrictions $(u^x_{h'})$ and deduce that
$$
(\rho^x_{h'} + a)\E^{n}\res A^x = J(u^x_{h'} + w^x) -\Gamma(u^x_{h'},w^x) - T_{u^x_{h'}} \wstar 
J(u^x + w^x) -\Gamma(u^x,w^x) - T_{u^x}
$$
in the sense of distributions. 
By Corollary \ref{c:detgeneral} and Proposition 
\ref{p:det} we are able to identify the weak limit in \eqref{eq:rx} 
\begin{equation}\label{eq:identification}
r^x = \det\nabla_y u^x +a = M_n\nabla u(x,\cdot)\res d\pi + a.
\end{equation}
We fix a a convex increasing function
 with superlinear growth $\varphi$ satisfying
\begin{equation}\label{eq:ambr6}
\lim_{t\rightarrow+\infty}\frac{\Psi(t)}{\varphi(t)}=+\infty.
\end{equation}
Using the previous convergence \eqref{eq:rx}, \eqref{eq:identification} on almost
 every slice and integrating with respect to $x$ we 
 deduce by the convexity of $\varphi$ that
$$
\int_A \varphi\left( \big|M_n\nabla u\res d\pi + a \big| \right)  d\Leb{m} \leq \liminf_h
\int_A \varphi(|\rho_h\res d\pi + a|) d\Leb{m} 
+\e\mu_{T_h,\pi}(A)+ \e\int_A|\nabla u_h|^p d\Leb{m}. 
$$
Adding this inequality on a finite number of disjoint
 open subsets $A_j$, with arbitrary choices of $a_i\in\R$, we obtain 
$$
\int_{\Omega} \varphi\left( \big| M_n\nabla u\res d\pi + \xi \big| \right) d\Leb{m} 
\leq \liminf_h
\int_{\Omega} \varphi (|\rho_h\res d\pi+\xi|) d\Leb{m} 
+\e\Sz(T_h) +\e\int_{\Omega}|\nabla u_h|^p d\Leb{m}, 
$$
where $\xi:=\sum_j a_j\chi_{A_j}$. Letting $\e\downarrow 0$ we 
can disregard the size and Dirichlet terms in the last inequality to
get
\begin{equation}\label{Psiineq}
\int_{\Omega} \varphi\left( \big| M_n\nabla u\res d\pi + \xi \big| \right) d\Leb{m} 
\leq \liminf_h
\int_{\Omega} \varphi (|\rho_h\res d\pi+\xi|) d\Leb{m}.
\end{equation}
Taking $\varphi_n(t):= \tfrac{\varphi(t)}{n}\vee t$, we have that
$\varphi_n$ are still convex, increasing, superlinear 
at infinity and satisfy \eqref{eq:ambr6}, therefore \eqref{Psiineq} is applicable with $\varphi=\varphi_n$. 
Given $\delta>0$ fix $C_{\delta}$ such that $\varphi_1(t)\leq\delta\Psi(t)$ for $t>C_{\delta}$;
 we also let $\Omega_{h,\delta} = \big\{| \rho_h\res d\pi + \xi | >C_{\delta}\big\}$. 
By applying \eqref{Psiineq} with $\varphi=\varphi_n$ we have therefore
\begin{align*}
 \int_{\Omega}\big| & M_n\nabla u\res d\pi + \xi \big| d\Leb{m}\leq 
 \int_{\Omega}\varphi_n\left(\big| M_n\nabla u\res d\pi + \xi \big| \right) d\Leb{m} \leq
\liminf_h \int_{\Omega}\varphi_n(|\rho_h\res d\pi + \xi|) d\Leb{m}  \\
& \leq \liminf_h \int_{\Omega}\big| \rho_h\res d\pi + \xi \big| 
+ \limsup_h \int_{\Omega_{h,\delta}}\varphi_1\left(\big| \rho_h\res d\pi + \xi \big| \right)
+ \sup_{0\leq t\leq C_\delta}\{\varphi_n(t) - t \}\Leb{m}(\Omega_{h,\delta}^c) \\
& \leq \liminf_h \int_{\Omega}\big| \rho_h\res d\pi + \xi \big| 
+ \limsup_h \delta \int_{\Omega_{h,\delta}}\Psi\left(\big| \rho_h\res d\pi + \xi \big| \right)
+ \sup_{0\leq t\leq C_\delta}\{\varphi_n(t) - t \}\Leb{m}(\Omega) .
 \end{align*}
Letting $n\rightarrow\infty$ the third term vanishes because 
 $\varphi_n(t)\downarrow t$ uniformly 
on compact sets. Eventually,  
sending $\delta\downarrow 0$ we obtain
\begin{equation}\label{dalmaso}
 \int_{\Omega}\big| M_n\nabla u\res d\pi + \xi \big| \leq \liminf_h
 \int_{\Omega}\big| \rho_h\res d\pi+ \xi\big|.
\end{equation}
Inequality \eqref{dalmaso} is actually valid for every $\xi\in L^1(\Omega)$
 by approximation, since the set of functions of type $\sum_j a_j\chi_{A_j}$
  is dense in $L^1$. 
We therefore address the last point \eqref{wregpartpi} 
thanks to Lemma \ref{l:dalmaso} below: the weak limit $\rho$ must 
be the equal to $M_n\nabla u$, the density of $R_u$ with respect to $\Leb{m}$.
\end{proof}

\begin{lemma}\label{l:dalmaso}
 Let $(z_h)\subset L^1(\Omega)$ be a weakly compact sequence
  and suppose that, for some $z\in L^1(\Omega)$, it holds
$$
\int_{\Omega}|z+\xi|\, d\Leb{m}\leq
\liminf_h\int_{\Omega}|z_h+\xi|\, d\Leb{m}\qquad\forall \xi\in L^1(\Omega).
$$
Then $z_h\w z$ weakly in $L^1(\Omega)$.
\end{lemma}
We refer to \cite{Amb89} for the proof.

\section{Applications}
\noindent
We now present an application of Theorem \ref{t:compactness} 
to a minimization problem. The choice of our Lagrangian is motivated
 by the introduction of a new functional of the calculus of variation,
  presented in \ref{ss:MS}, 
 aiming to generalize the classical Mumford-Shah energy 
 \cite{MumSha89,DeGCarLea89,AmbFusPal} to
 vector valued maps with singular set of codimension at least 2. 
The discussion in the introduction already mentioned the central
 role of the distributional jacobian in relation to low dimensional singularities: 
 in this model we replace the singularities of the derivative by the singularities
  of the jacobian and we measure them with the size functional of section~\ref{s:size}.

\subsection{Existence result for general Lagrangians} 
We fix an open, regular and bounded subset $\Omega$ of $\R^m$.
For approximately differentiable maps $u:\Omega\rightarrow\R^n$
we let $M\nabla u$ be the vector of all minors of $k\times k$ 
submatrices of $\nabla u$, with $k$ ranging from $1$ to $n$ and we let
$\kappa=\sum_{k=1}^n\binom{m}{k}\binom{n}{k}$ be its dimension. 
Given $w\in\R^\kappa$ we let $w_{\ell}$ the variables relative to
the $\ell\times\ell$ minors. We also denote $\mathscr{L}_m$ the 
$\sigma$-algebra of Lebesgue measurable subsets of $\R^m$ and 
$\mathscr{B}(\R^{n+\kappa})$ the $\sigma$-algebra of Borel subsets of $\R^{n+\kappa}$. 
 For the bulk part of the energy it is natural to treat polyconvex Lagrangians: 
 the lower semicontinuity properties of such energies with respect
  to the weak $W^{1,p}$ convergence for $p<n$ has been thoroughly
  studied, see \cite{CelDal94,FusHut95,FonLeoMal05,Mar86}.

\begin{theorem}[Existence of minimizers for polyconvex Lagrangians]\label{t:existlagr}
Assume $r,p$ satisfy $r<\infty$, $\tfrac 1r + \tfrac{n-1}{p}<1$ and
let $c>0$ be a positive constant. 
Let $f:\Omega\times\R^n\times\R^\kappa\rightarrow[0,+\infty)$ satisfy
 the following hypotheses:
\begin{itemize}
\item[(a)] $f$ is $\mathscr{L}_m\times\mathscr{B}(\R^{n+\kappa})$-measurable;
\item[(b)] for $\Leb{m}$-a.e. $x\in\Omega$, $(u,w)\mapsto f(x,u,w)$ is lower semicontinuous;
\item[(c)] for $\Leb{m}$-a.e. $x\in\Omega$, $w\mapsto f(x,u,w)$ is convex in $\R^\kappa$ for every $u\in\R^n$;
\item[(d)] 
$
f(x,u,w)\geq c\big(|u|^r + |w_1|^p + \Psi(|w_n|)\big)
$ 
for some function $\Psi$ satisfying the hypotheses of Theorem \ref{t:compactness}.
\end{itemize}
Let also $g:\Omega\rightarrow [c,+\infty)$ be a lower semicontinuous function.\\
Then, for every $u_0\in W^{1-\frac 1p,p}(\partial \Omega,\R^n)$ there exists
a solution to the problem
\begin{equation*}\label{eq:lagrangian}
\min_{u\in\GSB_n(\Omega),\,u=u_0\mbox{ on }\partial\Omega}\left\{ \int_{\Omega} f\big(x,u(x),M\nabla u(x)\big)\,dx +
\int_{\Omega\cap S_u}g(x)\,d\Haus{m-n}(x)\right\}. \tag{P}
\end{equation*}
\end{theorem}
\begin{proof}
Suppose the energy \eqref{eq:lagrangian} is finite for some function in $\GSB_n(\Omega)$ with
 trace $u_0$, otherwise there is nothing to prove. 
Pick a minimizing sequence $(u_h)$: by the growth assumption 
(d) there exist $u \in L^r\cap W^{1,p}$ and a subsequence
(not relabeled) such that 
\begin{equation}\label{eq:strongL1}
u_h\rightarrow u \mbox{ in }L^1 
\end{equation}
and $\nabla u_h\rightharpoonup\nabla u$ in $L^p$.
Since $\tfrac 1r + \tfrac{n-1}{p}<1$ we can choose $s<r$ 
 such that $\tfrac 1s + \tfrac{n-1}{p}\leq 1$: by Chebycheff's 
 inequality we know that $u_h\rightarrow u$ in $L^s$. 
Moreover we know that the absolutely continuous parts of $Ju_h$
 satisfy
\begin{equation*}
\int_{\Omega}\Psi\left(\big| M_n\nabla u_h\big|\right)\,dx \leq C,
\end{equation*}
and that by the lower bound $g\geq c$ we also have:
\begin{equation*}
\sup_h\Haus{m-n}(S_{u_h}\cap \Omega)<\infty.
\end{equation*}
Hence by the compactness Theorem \ref{t:compactness}, together
 with the classical Reshetnyak's Theorem for the minors of order less
  than $n$, we know that
\begin{equation}\label{eq:weakL1}
M\nabla u_h\rightharpoonup M\nabla u\quad\mbox{ weakly in }L^1.
\end{equation}
By \eqref{eq:strongL1} and \eqref{eq:weakL1} 
the lower semicontinuity result of Ioffe 
\cite{Iof77I,Iof77II} (see also \cite[Theorem 5.8]{AmbFusPal}) 
implies
$$
\liminf_h \int_{\Omega} f\big(x,u_h(x),M\nabla u_h(x)\big)\,dx 
\geq \int_{\Omega} f\big(x,u(x),M\nabla u(x)\big)\,dx.
$$
Finally, $g$ being lower semicontinuous,  
 the superlevel sets $\{g>t\}$ are open, hence
\begin{align*}
\liminf_h \int_{\Omega\cap S_{u_h}}g(x)\,d\Haus{m-n}(x) &=
\liminf_h\int_0^{+\infty} \Haus{m-n}(S_{u_h}\cap \{g>t\})\,dt \\
&\geq \int_0^{+\infty} \liminf_h \Haus{m-n}(S_{u_h}\cap \{g>t\})\,dt \\
&\geq \int_0^{+\infty}  \Haus{m-n}(S_{u}\cap \{g>t\})\,dt \\
&= \int_{\Omega\cap S_{u}}g(x)\,d\Haus{m-n}(x),
\end{align*}
because the size is lower semicontinuous on open sets, see \eqref{eq:lscA}.
\end{proof}
Recall that by the Sobolev embedding we can drop the growth
 condition on $u$ provided $p>\tfrac{mn}{m+1}$. 
 Notice also that we can formulate problem \eqref{eq:lagrangian} 
 and the corresponding boundary value condition 
 in a slightly different way, in order to include in the energy the possible
 appearance of singularities at the boundary. 
 Let $U\Supset\Omega$ be a bounded open subset of $\R^m$: 
 we formulate the minimization problem in the following way:
\begin{equation*}\label{eq:lagrclosure}
\min_{u\in\GSB_n(U), \,u=u_0\mbox{ in }U\setminus\Omega}
\left\{ \int_{U} f\big(x,u(x),M\nabla u(x)\big)\,dx +
\int_{U\cap S_u}g(x)\,d\Haus{m-n}(x)\right\}\tag{P'}
\end{equation*}
Every competitor being equal to $u_0$ in $U\setminus\overline{\Omega}$, 
problem \eqref{eq:lagrclosure} accounts for variations of $Ju$ in the
 closure $\overline{\Omega}$. Moreover Theorem \ref{t:existlagr} readily 
 applies to this case, as the condition $u=u_0$ in $U\setminus\Omega$ is 
 closed for the strong $L^1$ convergence. To explicit the dependence on the energy 
 on the datum $u_0$ and on the domain $U$  
  we adopt in the sequel the notation
 $$
 F(u,\Omega;u_0,U)
 $$
 for the energy in \eqref{eq:lagrclosure}.

\subsection{A Mumford-Shah functional of codimension higher than one}\label{ss:MS}
As anticipated in the beginning of the section the study of general 
functionals of the form \eqref{eq:lagrangian} was modeled on the 
Mumford-Shah type functional
\begin{equation}\label{eq:MS}
MS(u,\Omega):= \int_{\Omega} |u|^r + |\nabla u|^p + |M_n\nabla u|^\gamma\,dx + \Haus{m-n}(S_u\cap\Omega)
\end{equation}
defined on $\GSB_n(\Omega)$, 
with $r,p$ satisfying \eqref{exp}, $\gamma>1$, 
together with suitable boundary data. Theorem \ref{t:existlagr}
shows the existence of minimizers of \eqref{eq:MS} for both Dirichlet
 problems \eqref{eq:lagrangian} and \eqref{eq:lagrclosure}: 
 it is however desirable that at least for some boundary datum 
 $u_0$ the minimizer presents some singularity.
In the next proposition we show that this is the case:
\begin{proposition}[Nontrivial minimizers for $MS$, formulation \eqref{eq:lagrclosure}]\label{p:nontrivial}
Let $m=n$ and $u_0:B_2\rightarrow \R^n$ be the identity: 
$u_0(x)= x$. Then for $\e$ sufficiently small every minimizer
$u\in\GSB_n(B_2)$ of 
$$
MS_\e(u,B_1;x,B_2):=\int_{B_2} \e\big(|u|^r + |\nabla u|^p \big) +|\det\nabla u|^\gamma\,dx + \e\Haus{0}(S_u\cap B_2)
$$ 
such that $u(x)= x$ in $B_2\setminus B_1$ must satisfy
$$
S_u \cap\overline{B_1}\neq \emptyset.
$$
\end{proposition}
\begin{proof}
We show that for every competitor $v$ with $\|Jv\|\ll\Leb{n}$
 and for $\e$ small enough it holds:
 $$
 MS_\e(v,B_1;x,B_2)>MS_\e(w,B_1;x,B_2),
 $$
where
\begin{eqnarray*}
w(x)=
\left\{
\begin{array}{ll}
\tfrac{x}{|x|} & \mbox{ in }B_1,\\
x   &\mbox{ in }B_2\setminus B_1.
\end{array}
\right.
\end{eqnarray*}
For the rest of the proof $c$ will denote a generic positive constant
 we do not keep track of. Let us compute the energy of $\tfrac{x}{|x|}$:
  the Dirichlet and $L^r$ parts are simply constants. Moreover
\begin{equation*}
\det \nabla  w = \chi_{B_2\setminus B_1} \quad\mbox{ and }\quad S_{ w} =\{0\}.
\end{equation*}
Hence $MS_\e( w,B_1;x,B_2) = c\e + \Leb{n}(B_2\setminus B_1)$. 
On the contrary for almost every radius $\rho$ it holds:
$$
\int_{B_\rho} \det\nabla v\,dx = \int_{\partial B_\rho} v^1 dv^2\wedge\dots\wedge dv^n
$$
(see \cite[Lemma 2.1]{DelGhi10} for a simple proof of this fact). 
Since $u(x) = x$ outside $B_1$ for almost every 
$\rho \in (1,2)$ we have $\int_{B_\rho} \det\nabla v\,dx = \Leb{n}(B_{\rho})$, 
hence by Jensen's inequality
$$
\int_{B_\rho} |\det\nabla v|^\gamma\,dx \geq \Leb{n}(B_{\rho}).
$$
Summing up:
$$
MS_\e(v,B_1;x,B_2) \geq \int_{B_\rho} |\det\nabla v|^\gamma\,dx 
\geq \Leb{n}(B_{\rho}) > c\e + \Leb{n}(B_2\setminus B_1) =
MS_\e(w,B_1;x,B_2)
$$
choosing first $\rho$ sufficiently close to 2 and then $\e$ sufficiently small. 
Therefore the minimizer $u$ must have a nonempty
singular set $S_u$,
and since $u$ is linear in the open set $B_2 \setminus B_1$, the singularity must
 be in $\overline{B_1}$.
\end{proof}

It is easy to generalize the same proposition to the case 
$m\geq n$, by simply taking the trivial extension in the extra
 variables and showing that every minimizer has a nontrivial singular
  set. In analogy with \cite{HarLinWan98}, we expect however the 
  singularities to appear in the interior.

The argument in Proposition \ref{p:nontrivial} essentially
 exploits the presence of the jacobian term: this is not 
 coincidental, as the next proposition shows. Recall that
 the sum of a $GSB_nV$ function and a $C^1$ function is
 again in $GSB_nV$.
\begin{proposition}
Every local minimizer of
$$
u\in GSB_nV(\Omega)\mapsto \int_{\Omega}|\nabla u|^p\,dx + \Haus{m-n}(S_u\cap\Omega)
$$ 
is locally of class $C^{1,\alpha}$ in $\Omega$.
\end{proposition}
\begin{proof}
It is sufficient to perform an outer variation of the minimizer $u$ 
along a $\phi\in C^1_c(\Omega,\R^n)$ map: $\e\mapsto u +\e\phi$ 
and apply Corollary \ref{c:detgeneral} to obtain that 
$$
S_{u + \e\phi} = S_u.
$$
Hence the size term is constant and $u$ satisfies:
$$
\int_{\Omega} |\nabla u|^{p-2}\nabla u \nabla \phi\,dx = 0 \quad \forall \phi\in C^1_c(\Omega,\R^n).
$$
Therefore $u$ is a $p$-harmonic $W^{1,p}$ function,
hence $u\in C^{1,\alpha}_\loc$ by \cite{Uhl77, DiB83}.
\end{proof}

\subsection{Traces}\label{ss:traces}
 In the spirit of solving \eqref{eq:lagrclosure}, 
 the nonuniqueness Example \ref{ex:cerchi} below raises the problem 
 of the dependence of the energy on the extension $u_0:U\setminus\Omega\rightarrow\R^n$
  to a given Sobolev trace $u\sbarretta_{\partial\Omega}$.
The example was communicated to us by C. De Lellis, see
 also Examples 1 and 2 in Section 3.2.5 of \cite{GiaModSou}, 
 and the discussion on weak and strong anchorage condition therein. 
It shows that if we want to detect the presence of singularities
 of $Ju$ at the boundary of $\Omega$, the Sobolev trace is not
  sufficient to characterize it.
\begin{example}[Singularity at the boundary]\label{ex:cerchi}
Let $u:\R^2\rightarrow S^1$ be defined by
\begin{equation}\label{eq:cerchi}
u(x,y)=\left( \frac{y^2- (x-1)^2}{(x-1)^2+y^2},\frac{2(1-x)y}{(x-1)^2+y^2}\right).
\end{equation}
This map represents the normal unit vectorfield 
of the family of circles centered on the real axis and tangent
 to $S^1$ in the point $(1,0)$. If $\theta$ is the angle that
 the vector $(x-1,y)$ makes with the real axis, we can write 
 $u(x,y)=(-\cos(2\theta),-\sin(2\theta))$, hence by Example
  \eqref{ex:monopole} $Ju=2\pi\llbracket (1,0)\rrbracket$.  Note that $u$ is the identity
   map when restricted to $S^1$. 
Nonetheless we can construct another map $\tilde{u}$
\begin{equation}\label{eq:cerchitilde}
\tilde{u}(x,y)=\left\{
\begin{array}{ll}
u(x,y) \qquad &\mbox{ for }\quad |x|<1,\\
 \Big(\frac{x}{\sqrt{x^2+y^2}},\frac{y}{\sqrt{x^2+y^2}}\Big)\qquad &\mbox{ for }\quad |x|\geq 1.
\end{array}
\right. 
\end{equation}
In this case, by Example~\ref{ex:monopole}, $J\tilde{u}=\pi\llbracket (1,0)\rrbracket$. 
Hence $u\sbarretta_{B_1}$ admits two different Sobolev extensions
$u$ and $\tilde{u}$ sharing the same trace at the boundary
 but whose jacobians are different in $\overline{\Omega}$: the
 trace of a Sobolev function does not characterize the jacobian
 $Jv\res{\partial\Omega}$ of all the possible extensions $v$.
\end{example}

It is interesting to know when part of the distributional jacobian
 can be represented as a boundary integral. Recall that 
the slicing Theorem \ref{t:slicing} already provides an answer
 to this question, 
because if $u:\Omega\rightarrow\R^n$ then 
$\partial (j(u)\res\{\pi>t\}) = Ju\res\{\pi>t\} + \langle j(u),\pi,t\rangle$, 
 where $\pi$ is the distance from $\partial\Omega$. 
 However, as Example \ref{ex:cerchi} shows, this statement holds only 
 for $\Leb{1}$-a.e. $t$. The following
  proposition improves the general result by slicing, under additional
   hypotheses on the summability of $u$ and of its trace. 
This result is already present in the literature, 
see \cite[Vol. I, p. 274]{GiaModSou} and \cite[Lemma 6.1]{AlbBalOrl05}: 
we report the proof for the reader's convenience. 
Denote for simplicity $g(u):=u^1 du^2\wedge\dots\wedge du^n$. 

\begin{proposition}[Stokes' Theorem]\label{p:stokes}
If $u\in W^{1,n}(\Omega,\R^n)$ and $u\sbarretta_{\partial\Omega}\in W^{1,n-1}(\partial\Omega, \R^n) \cap L^{\infty}(\partial\Omega,\R^n)$ then 
Stokes' theorem holds: 
$$
\partial (j(u)\res\Omega) = Ju\res\Omega + \langle j(u),\partial\Omega\rangle
$$
with the representation
$$
\langle j(u),\partial\Omega\rangle(\omega) = \int_{\partial\Omega} \langle g(u),\tau_{\partial\Omega}\rangle\omega\,d\Haus{n-1}
$$
where $\tau_{\partial\Omega}$ orients $\partial\Omega$ as the boundary of $\Omega$. 
In particular $\langle j(u),\partial\Omega\rangle$ depends
only on the trace $u\sbarretta_{\partial\Omega}$.
\end{proposition}
\begin{proof}
Suppose for simplicity that $\Omega = \R^{n}_+ = \R^n\cap\{x^n>0\}$, $\spt(u)\subset B_1$ and let $\phi:\R^{n-1}\rightarrow \R$ be a positive convolution kernel with compact support in $\R^{n-1}$. Set
$$
u_\e(x',x^n) = \frac{1}{\e^{n-1}}\int_{\R^{n-1}} u(x'-y', x^n)\phi\left(\frac{x'-y'}{\e}\right)\,dy':
$$
since the convolution in the $x'$ variables commutes with the trace operator
we still have $u_\e\sbarretta_{\R^{n-1}}(x') = u_\e(x',0)$; moreover $u_\e(\cdot,0)\in C^1(\R^{n-1},\R^n)$ and the following estimates hold:
\begin{equation}\label{eq:convol}
\|u_\e\|_{W^{1,n}(\R^n_+,\R^n)} \leq \|u\|_{W^{1,n}(\R^n_+,\R^n)},\qquad \|u_\e(\cdot,0)\|_{W^{1,n-1}(\R^{n-1},\R^n)} \leq \|u(\cdot,0)\|_{W^{1,n-1}(\R^{n-1},\R^n)}
\end{equation}
and since the translations are strongly continuous in $L^p$, 
\begin{equation}\label{eq:convolest}
\|u_\e-u\|_{W^{1,n}(\R^n_+,\R^n)} + \|u_\e(\cdot,0)-u(\cdot,0)\|_{W^{1,n-1}(\R^{n-1},\R^n)}\rightarrow 0.
\end{equation}
We claim that Stokes' Theorem holds for $u_\e$: for every 
$\omega\in \DD^0(\R^n)$ 
\begin{equation}\label{eq:stokes}
\partial(j(u_\e)\res\R^n_+)(\omega) = \int_{\R^n_+}\omega\det\nabla u_\e\,dx + 
\int_{\R^{n-1}\times\{0\}} \omega g(u_\e(\cdot,0)).
\end{equation}
In fact extending $u_\e(x',x^n):=u_\e(x',0)$ for $x^n \in [-1,0]$ and then convolving
with a smooth kernel $\rho_\delta$ supported in $B_\delta(0)$ we obtain a smooth 
 $u_{\e,\delta}\in C^{\infty}(\R^n,\R^n)$ such that $\spt(u_{\e,\delta})\subset B_2\times [-2,2]$,  
\begin{eqnarray}
& u_{\e,\delta}(x',x^n) \rightarrow u_\e(x',0) \qquad &\mbox{ in  } C^1_\loc(\R^{n}_-,\R^n), \notag\\
& u_{\e,\delta}\rightarrow u_\e\qquad&\mbox{ in }W^{1,n}_\loc(\R^{n-1}\times(-1,+\infty),\R^n).\label{eq:convdelta}
\end{eqnarray}
More precisely it holds: $u_{\e,\delta}(x',-\delta) \rightarrow u_\e(x',0)$ 
in $C^1(\R^{n-1},\R^n)$. 
Hence
$$
\partial(j(u_{\e,\delta})\res\{x^n>-\delta\})(\omega) =
\int_{\{x^n>-\delta\}}\omega\det\nabla u_{\e,\delta}\,dx +
\int_{\R^{n-1}\times\{-\delta\}} \omega g(u_{\e,\delta}(\cdot,-\delta)):
$$
letting $\delta\downarrow 0$ the left hand side converges to $\partial(j(u_\e)\res\R^n_+)(\omega)$
by \eqref{eq:convdelta}. The boundary term in right hand side tends to 
$$
\int_{\R^{n-1}\times\{0\}} \omega(\cdot,0) g(u_\e(\cdot,0)) 
$$ 
because the convergence is $C^1$ and $\omega$ is smooth. Regarding the volume integral 
we can estimate
$$
|\nabla u_{\e,\delta}(x)| =| (\rho_\delta*\nabla u_\e)(x)| \leq
 \|u_\e(\cdot,0)\|_{C^1} +  \fint_{B_\delta(x)\cap\{y^n>0\}}|\nabla u_\e(y)|\,dy
$$
 hence
\begin{align*}
\left| \int_{  \{|x^n|<\delta\}}\omega\det\nabla u_{\e,\delta}\,dx \right| 
&\leq \|\omega\|_{C^0} \int_{ \{|x^n|<\delta\}}|\nabla u_{\e,\delta}|^n\,dx  \\
 &\leq  c_n\|\omega\|_{C^0} \left( \|u_\e(\cdot,0)\|_{C^1}^n\delta+ 
 \int_{ \{|x^n|<\delta\}}\fint_{B_\delta(x)\cap\{y^n>0\}}|\nabla u_\e(y)|^n\,dy\,dx \right) \\
 &\leq  c_n\|\omega\|_{C^0} \left( \|u_\e(\cdot,0)\|_{C^1}^n\delta+ 
 \int_{ \{0<x^n<2\delta\}}|\nabla u_\e(x)|^n\,dx \right) \rightarrow 0.
\end{align*}
Clearly $ \int_{  \{ x^n>\delta\}}\omega\det\nabla u_{\e,\delta}\,dx  \rightarrow
\int_{  \{ x^n>0\}}\omega\det\nabla u_{\e}\,dx $, therefore \eqref{eq:stokes} is true.

We now want to pass to the limit for $\e\downarrow 0$ in \eqref{eq:stokes}. 
The left hand side goes to $\partial(j(u)\res\R^n_+)(\omega) $ because of \eqref{eq:convolest}; 
similarly for the volume term. Regarding the boundary term the convergence of
 the minors on the slice needs to be improved. 
The estimates \eqref{eq:convol} and the classical result \cite{CoiLioMeySem93} gives
a uniform bound of the Hardy norm  \cite[Chapter IV]{Ste} of the minors of order $n-1$: 
$$
\|du^2_\e(\cdot,0)\wedge\dots\wedge du^n_\e(\cdot,0)\|_{\Hardy(\R^{n-1},\R^n)}\leq \|u(\cdot,0)\|^{n-1}_{W^{1,n-1}(\R^{n-1},\R^n)}.
$$
We already know from Reshetnyak's Theorem that 
 $
 du^2_\e(\cdot,0)\wedge\dots\wedge du^n_\e(\cdot,0)\wstar du^2(\cdot,0)\wedge\dots\wedge du^n(\cdot,0)$
  in the sense of distributions; moreover 
smooth functions are dense in $VMO(\R^{n-1},\R^n)$ and $VMO^* = \Hardy$,
 so 
$$
du^2_\e(\cdot,0)\wedge\dots\wedge du^n_\e(\cdot,0)\wstar 
du^2(\cdot,0)\wedge\dots\wedge du^n(\cdot,0)\qquad\mbox{ in}\quad\sigma(\Hardy,VMO).
$$
Finally the trace $u_\e(\cdot,0)$ belongs to
 $$W^{1-\frac{1}{n},n}(\R^{n-1},\R^n)\subset VMO(\R^{n-1},\R^n)$$ (see \cite[Theorem 7.58]{Ada}, \cite[Example 2]{BreNir95} for the inclusions). Hence $\|u_\e(\cdot,0) - u(\cdot,0)\|_{VMO}\rightarrow 0$ strongly and we can pass to the limit in
 \eqref{eq:stokes}
 $$
 \int_{\R^{n-1}} \omega g(u_\e(\cdot,0)) \rightarrow \int_{\R^{n-1}} \omega g(u(\cdot,0)).
 $$
By \eqref{eq:convolest} also the left hand side of \eqref{eq:stokes} converges to $\int_{\R^n_+}\omega\det\nabla u\,dx$.
\end{proof}

 In the example 
  above the smooth extension $\tilde{u}$ is certainly preferable to $u$, 
  where an ``extra" singularity comes from the outside. 
  A partial answer 
 to this problem can be given if we assume a better differentiability of the 
 outer extension, up to $\partial \Omega$: 
 
\begin{proposition}[See {\cite[Vol. I, p.266]{GiaModSou}}]\label{p:uniqext}
Let $v,w\in L^s\cap W^{1,p}(U,\R^n)$ satisfy the following conditions:
\begin{itemize}
\item $v\sbarretta_{\,\Omega} = w\sbarretta_{\,\Omega}$;
\item $v\sbarretta_{\,U\setminus \Omega}, w\sbarretta_{\,U\setminus \Omega} \in W^{1,n}(U\setminus\Omega,\R^n)$;
\item $v\sbarretta_{\,\partial \Omega}=
w\sbarretta_{\,\partial \Omega} \in W^{1,n-1}(\partial \Omega,\R^n)$.
\end{itemize}
Then:
$$
Jv-Jw = \big( \det \nabla v -\det\nabla w\big)\E^n\res  (U\setminus\Omega).
$$
\end{proposition}
\begin{proof}
We can write
\begin{align}
Jv &= \partial j(v) = \partial (j(v)\res\Omega) + \partial (j(v)\res(U\setminus \Omega)) 
= \partial (j(v)\res\Omega) + Jv\res (U\setminus \Omega ) -  \langle j(v),\partial\Omega\rangle.
\end{align}
Subtracting the analogous expression for $Jw$ we obtain
$$
Jv-Jw = (Jv - Jw)\res (U\setminus \Omega)  - \langle j(v) -j(w) ,\partial\Omega\rangle= 
\big( \det \nabla v -\det\nabla w\big)\E^n\res  (U\setminus\Omega)
$$
because Proposition \ref{p:stokes}  
applied to the open set $U\setminus\Omega$ implies that $v\sbarretta_{\partial\Omega} = w\sbarretta_{\,\partial\Omega}$, hence $\langle g(v) - g(w),\tau_{\,\partial\Omega}\rangle = 0$.
\end{proof}

Therefore, if we aim at formulating problem \eqref{eq:lagrclosure}
in a local way, that is depending only on the values of $u$ in
 $\overline{\Omega}$, at least when the trace is sufficiently ``nice",
 we can proceed as follows. 
If $u\sbarretta_{\partial\Omega}$ belongs to $W^{1,n-1}$
and admits a $W^{1,n}$ extensions outside $\Omega$, 
we can conventionally agree to pick one of such extensions
 to $U\setminus\Omega$: 
the result of Proposition \ref{p:uniqext}
implies that the jacobian in $\overline{\Omega}$ of every competitor
does not depend on the particular choice we made.
Note however that the smoothness of the trace 
 does not imply membership of the extension to $\GSB_n(U)$. In fact,
 it is sufficient to place the infinite dipoles of the function $g$ in 
 Example \ref{ex:dipole} so that the singularities lie on $\partial B_1$
 and do not overlap. The constant extension outside the ball  
 provides a map whose jacobian has both infinite mass and size.
 
 In conclusion, in order to solve Problem \eqref{eq:lagrclosure}
it seems necessary to impose membership of the competitors 
to $GSB_nV(U)$, while for a fairly broad class of boundary data
 the energy in $\overline{\Omega}$ shall not depend on the particular extension.

\end{document}